\documentclass[11pt,reqno]{amsart}

\usepackage{fullpage, amsmath, amsthm, amstext, amssymb, amsfonts, fancyhdr, graphicx, tikz, forest, mathrsfs, array, mathtools, bm}
\usepackage[normalem]{ulem}
\usetikzlibrary{graphs,graphs.standard,quotes}
\usetikzlibrary{matrix,positioning,quotes}

\definecolor{myblue}{RGB}{80,80,160}
\definecolor{mygreen}{RGB}{80,160,80}

\numberwithin{equation}{section}

\newtheorem{theorem}{Theorem}[section]
\newtheorem{lemma}[theorem]{Lemma}

\newtheorem{example}[theorem]{Example}
\newtheorem{definition}[theorem]{Definition}
\newtheorem{proposition}[theorem]{Proposition}

\theoremstyle{remark}

{\bf}{\rm}

\DeclareMathOperator{\PF}{\text{PF}}
\DeclareMathOperator{\IPF}{\text{IPF}}

\newcommand{\aR}{\mathbf a}
\newcommand{\bR}{\mathbf b}
\newcommand{\cR}{\mathbf c}
\newcommand{\sR}{\mathbf s}
\newcommand{\kR}{\mathbf k}
\newcommand{\xR}{\mathbf x}
\newcommand{\yR}{\mathbf y}
\newcommand{\uR}{\mathbf u}
\newcommand{\vR}{\mathbf v}
\newcommand{\rR}{\mathbf r}
\newcommand{\piR}{\boldsymbol \pi}
\newcommand{\tauR}{\boldsymbol \tau}
\newcommand{\lambdaR}{\boldsymbol \lambda}
\newcommand{\alphaR}{\boldsymbol \alpha}
\newcommand{\betaR}{\boldsymbol \beta}
\newcommand{\gammaR}{\boldsymbol \gamma}
\newcommand{\deltaR}{\boldsymbol \delta}

\newcommand{\Sym}{\mathfrak{S}}
\newcommand{\MS}{\text{MS}}

\newcommand{\ER}{\mathbb E}

\newcommand{\F}{\mathscr F}
\newcommand{\C}{\mathscr C}

\newcommand{\Var}{\mathbb V \mathrm{ar}}
\newcommand{\Cov}{\mathbb C \mathrm{ov}}
\newcommand{\be}{\begin{equation}}
\newcommand{\ee}{\end{equation}}

\newcommand{\old}[1]{}

\NewDocumentCommand\DownArrow{O{2.0ex} O{black}}{%
   \mathrel{\tikz[baseline] \draw [<->, line width=0.5pt, #2] (0,0) -- ++(0,#1);}
}

\begin{document}

\forestset{parent color/.style args={#1}{
    {fill=#1},
    for tree={fill/. wrap pgfmath arg={#1!##1}{1/level()*80},draw=#1!80!black}},
    root color/.style args={#1}{fill={{#1!60!black!25},draw=#1!80!black}}
}

\title{Parking functions: Interdisciplinary connections}

\author{Mei Yin}
\thanks{MY was supported by the University of Denver's Faculty Research Fund 84688-145601.}
\address{Department of Mathematics, University of Denver, Denver, CO 80208}
\email{mei.yin at du.edu}

\date{\today}


\begin{abstract}
Suppose that $m$ drivers each choose a preferred parking space in a linear car park with $n$ spots. In order, each driver goes to their chosen spot and parks there if possible, and otherwise takes the next available spot if it exists. If all drivers park successfully, the sequence of choices is called a parking function. Classical parking functions correspond to the case $m=n$.

We investigate various probabilistic properties of a uniform parking function. Through a combinatorial construction termed a parking function multi-shuffle, we give a formula for the law of multiple coordinates in the generic situation $m \lesssim n$. We further deduce all possible covariances, between two coordinates, between a coordinate and an unattempted spot, and between two unattempted spots. This asymptotic scenario in the generic situation $m \lesssim n$ is in sharp contrast with that of the special situation $m=n$.

A generalization of parking functions called interval parking functions is also studied, in which each driver is willing to park only in a fixed interval of spots. We construct a family of bijections between interval parking functions with $n$ cars and $n$ spots and edge-labeled spanning trees with $n+1$ vertices and a specified root.
\end{abstract}

\maketitle

\section{Introduction}
\label{intro}
Parking functions are an established area of research in combinatorics, with connections to labeled trees and forests (Chassaing and Marckert, \cite{CM}), non-crossing partitions and hyperplane arrangements (Stanley, \cite{Stanley1} \cite{Stanley2}), symmetric functions (Haiman, \cite{H}), abelian sandpiles (Cori and Rossin, \cite{CR}), and other topics.

Consider a parking lot with $n$ parking spots placed sequentially along a one-way street. A line of $m \leq n$ cars enters the lot, one by one. The $i$th car drives to its preferred spot $\pi_i$ and parks there if possible; if the spot is already occupied then the car parks in the first available spot after that. The list of preferences $\piR=(\pi_1,\dots,\pi_m)$ is called a \emph{generalized parking function} if all cars successfully park. (This generalizes the term \emph{parking function} which classically refers to the case $m=n$. When there is no risk of confusion we will drop the modifier ``generalized" and simply refer to both of these cases as parking functions). We denote the set of parking functions by $\PF(m, n)$, where $m$ is the number of cars and $n$ is the number of parking spots. The total number of parking functions is $|\PF(m, n)|=(n-m+1)(n+1)^{m-1}$ (Pitman and Stanley, \cite{PS1}). Using the pigeonhole principle, we see that a parking function $\piR\in \PF(m, n)$ must have at most one value $=n$, at most two values $\geq n-1$, and for each $k$ at most $k$ values $\ge n-k+1$, and any such function is a parking function. Equivalently, $\piR$ is a parking function if and only if
\begin{equation}\label{pigeon}
\#\{k: \pi_k \leq i\} \geq m-n+i, \hspace{.2cm} \forall i=n-m+1, \dots, n.
\end{equation}
Note that parking functions are invariant under the action of $\Sym_m$ by permuting cars.

In our previous work \cite{KY}, we investigated various probabilistic properties of a parking function chosen uniformly at random from $\PF(m, n)$, giving a formula for the law of a single coordinate. Adapting known results on random linear probes, we further deduced the covariance between two coordinates in the special situation $m=n$. This paper will delve deeper into the properties of a uniform parking function in the generic situation $m \lesssim n$. Our probabilistic results rely on an original combinatorial construction which we term a \emph{parking function multi-shuffle}, and our novel asymptotic calculation utilizes the multi-dimensional Cauchy product of the \emph{tree function} $F(z) = \sum_{s=0}^\infty (s+1)^{s-1}\frac{z^s}{s!}$, a variant of the Lambert function, and its generalizations. We will give all moments of multiple coordinates and deduce all possible covariances, between two coordinates, between a coordinate and an unattempted spot, and between two unattempted spots.

The multi-shuffle construction allows us to compute the number of parking functions $\PF(m, n)$ where the parking preferences of $l\leq m$ cars are arbitrarily specified. Alternatively, due to permutation symmetry, we can think that $l$ spots are already taken along a one-way street with $n$ parking spots, and we want to count the possible preferences for the remaining $m-l$ cars so that they can all successfully park. In the parking function literature, the set of successful preference sequences of the $m-l$ cars that enter the street later is referred to as \emph{parking completions} for $\tauR=(\tau_1, \dots, \tau_l)$ where the entries of $\tauR$ denote the $l$ spots that are taken previously, arranged in increasing order.

This parking scenario as well as its variations, such as defective parking functions where some drivers fail to park (Cameron et al., \cite{CJPS}), have generated significant interest over the years. Much progress has been made for the special case $m=n$ of parking functions. Parking completions with a single spot taken ($\tauR=(\tau_1)$ arbitrary) were enumerated by Diaconis and Hicks \cite{DH}. The case that the taken spots consist of a contiguous block starting from the first spot in the linear car park, $\tauR=(1, \dots, l)$, was first considered by Yan \cite{Yan1}, with an explicit formula given in a follow-up work by Gessel and Seo \cite{GS}. The formula was generalized by Ehrenborg and Happ \cite{EH} taking into account cars of different sizes. More recently, Adeniran et al. \cite{A} unified prior work on parking completions for $\PF(n, n)$ and computed the number of parking functions $\PF(n, n)$ where the parking preferences of $l\leq n$ cars are arbitrarily specified utilizing a pair of operations termed Join and Split. The multi-shuffle construction introduced in this paper builds upon our prior single shuffle construction \cite{KY} and is a further generalization to the above mentioned work by being applicable for general $m$ and $n$. Recognizing that unattempted parking spots break up a parking function into non-interacting pieces, the multi-shuffle construction also sheds light on the correlation between the coordinates of parking functions and unattempted spots.

Given a positive-integer-valued vector $\uR=(u_1, \dots, u_m)$ with $u_1 \leq \cdots \leq u_m$, a $\uR$-parking function of length $m$ is a sequence $(\pi_1, \dots, \pi_m)$ of positive integers whose non-decreasing rearrangement $(\lambda_1, \dots, \lambda_m)$ satisfies $\lambda_i\leq u_i$ for all $1\leq i\leq m$. Via a switch of coordinates in (\ref{pigeon}), we see that the parking function $\PF(m, n)$ investigated in this paper may be alternatively posed as a $\uR$-parking function, where the vector $\uR$ is an arithmetic progression: $\uR=(n-m+1, \dots, n)$. As we will see in Section \ref{completion}, more generally, a parking completion for $\PF(m, n)$ may be interpreted as a $\uR$-parking function, where the vector $\uR$ need not consist of consecutive numbers. Knowledge on $\PF(m, n)$ with specified parking preferences of $l\leq m$ cars therefore adds to the understanding of $\uR$-parking functions as well. In particular, our enumeration of parking completions provides a different perspective on the volume formula for Pitman-Stanley polytopes \cite{PS1}, and our mixed moment calculations for multiple coordinates of parking functions extend that of Kung and Yan \cite{KY2}, where the explicit formulas for the first and second factorial moments and a general form for the higher factorial moments of sums of $\uR$-parking functions were given.

This paper is organized as follows. Section \ref{pfs} illustrates the notion of parking function multi-shuffle that decomposes a parking function into smaller components (Definition \ref{shuffle}). This construction leads to an explicit characterization of multiple coordinates $\pi_1, \dots, \pi_l\in[n]$ of parking functions (Theorems \ref{main1} and \ref{component}). When $\pi_1, \dots, \pi_l$ consist of a contiguous block, a simplified characterization is given in Proposition \ref{un}. Section \ref{random} uses the multi-shuffle construction introduced in Section \ref{pfs} to investigate various properties of a parking function chosen uniformly at random from $\PF(m, n)$. We compute asymptotics of all moments of multiple coordinates in Theorem \ref{general-mean} in the generic situation $m \lesssim n$ and give complete technical details for all moments of two coordinates (Theorem \ref{mean}). We further derive all possible covariances concerning coordinates of parking functions and unattempted spots in Propositions \ref{two-coordinates}, \ref{coordinate-spot}, and \ref{two-spots}. The asymptotic scenario in the generic situation $m \lesssim n$ is contrasted with that of the special situation $m=n$ in Section \ref{special}. Finally, Section \ref{ipf} studies a generalization of parking functions called interval parking functions, in which each driver is willing to park only in a fixed interval of spots. We construct a family of bijections between interval parking functions $\IPF(n, n)$ and edge-labeled spanning trees $\F^e(n+1)$ (Theorem \ref{edge-labeled}).

\subsection*{Notations} Let $\mathbb{N}$ be the set of non-negative integers. For $m, n \in \mathbb{N}$, we write $[m, n]$ for the set of integers $\{m, \dots, n\}$ and $[n]=[1, n]$. For vectors $\aR, \bR \in [n]^m$, denote by $\aR\leq_C \bR$ if $a_i\leq b_i$ for all $i\in [m]$; this is the component-wise partial order on $[n]^m$. In a similar fashion, denote by $\aR <_C \bR$ if $a_i\leq b_i$ for all $i\in [m]$ and there is at least one $j \in [m]$ such that $a_j<b_j$. For $\bR \in [n]^m$, we write $[\bR]$ for the set of $\aR \in [n]^m$ with $\aR \leq_C \bR$. The conjugate (or reverse complement) of $\xR\in[n]^m$ is the vector $\xR^*=(n+1-x_m, \dots, n+1-x_1)$.

\section{Parking function multi-shuffle}
\label{pfs}
In this section we explore the properties of parking functions through a parking function multi-shuffle construction. We will write our results in terms of parking coordinates $\pi_1, \dots, \pi_l$ for explicitness, where $1\leq l \leq m$ is any integer. But due to permutation symmetry, they may be interpreted for any coordinates. Temporarily fix $\pi_{l+1}, \dots, \pi_m$. Let
\begin{equation}
A_{\pi_{l+1}, \dots, \pi_m}=\{\uR=(u_1, \dots, u_l): (u_1, \dots, u_l, \pi_{l+1}, \dots, \pi_m)\in \PF(m, n)\}.
\end{equation}
Via a switch of coordinates in (\ref{pigeon}), we see that $\piR=(u_1, \dots, u_l, \pi_{l+1}, \dots, \pi_{m})\in \PF(m, n)$ if and only if its non-decreasing rearrangement $\lambdaR=(\lambda_1, \dots, \lambda_m)$ satisfies $\lambda_i \leq n-m+i$ for all $1\leq i\leq m$. From the parking scheme, we may assume that $\uR=(u_1, \dots, u_l)$ is in strict increasing order, so that $u_i=\lambda_j \geq \lambda_i$ for some $j\geq i$. This implies that if $A_{\pi_{l+1}, \dots, \pi_m}$ is non-empty, then there is a unique maximal element (in component-wise partial order) $\uR \in [n]^l$ with $u_i \geq n-m+i$ for all $1\leq i\leq l$ and $A_{\pi_{l+1}, \dots, \pi_m}=[\uR]$. Therefore given the last $m-l$ parking preferences, it is sufficient to identify the largest feasible first $l$ preferences (if exists).

\begin{example}
Take $m=4$, $n=6$, $\pi_3=2$, and $\pi_4=6$. Then $A_{\pi_3, \pi_4}=[\uR]=[(4, 5)]$.
\end{example}

\begin{definition}\label{shuffle}
Take $1\leq l \leq m$ any integer. Let $\uR=(u_1, \dots, u_l) \in [n]^l$ be in increasing order with $u_i\geq n-m+i$ for all $1\leq i\leq l$. Say that $\pi_{l+1}, \dots, \pi_m$ is a \emph{parking function multi-shuffle} of $l+1$ parking functions $\alphaR_1 \in \PF(m-n+u_1-1, u_1-1), \alphaR_2 \in \PF(u_2-u_1-1, u_2-u_1-1), \dots, \alphaR_l \in \PF(u_l-u_{l-1}-1, u_l-u_{l-1}-1)$, and $\alphaR_{l+1} \in \PF(n-u_l, n-u_l)$ if $\pi_{l+1}, \dots, \pi_m$ is any permutation of the union of the $l+1$ words $\alphaR_1, \alphaR_2+(u_1, \dots, u_1), \dots, \alphaR_{l+1}+(u_l, \dots, u_l)$. We will denote this by $(\pi_{l+1}, \dots, \pi_m) \in \MS(m-n+u_1-1, u_1-1, u_2-u_1-1, \dots, u_l-u_{l-1}-1, n-u_l)$.
\end{definition}

\begin{example}
Take $m=8$, $n=10$, $u_1=6$, and $u_2=8$. Take $\alphaR_1=(2, 1, 2) \in \PF(3, 5)$, $\alphaR_2=(1) \in \PF(1, 1)$, and $\alphaR_3=(2, 1) \in \PF(2, 2)$. Then $(2, \overline{7},2, \underline{9}, \underline{10},1) \in \MS(3, 5, 1, 2)$ is a multi-shuffle of the three words $(2, 1, 2)$, $(7)$, and $(10, 9)$.
\end{example}

\begin{theorem}\label{main1}
Take $1\leq l \leq m$ any integer. Let $\uR=(u_1, \dots, u_l) \in [n]^l$ be in increasing order with $u_i\geq n-m+i$ for all $1\leq i\leq l$. Then $A_{\pi_{l+1}, \dots, \pi_m}=[\uR]$ if and only if $(\pi_{l+1}, \dots, \pi_m) \in \MS(m-n+u_1-1, u_1-1, u_2-u_1-1, \dots, u_l-u_{l-1}-1, n-u_l)$.
\end{theorem}

\begin{proof}
``$\Longrightarrow$'' $A_{\pi_{l+1}, \dots, \pi_m}=[\uR]$ is equivalent to saying that $\piR=(u_1, \dots, u_l, \pi_{l+1}, \dots, \pi_m)$ is a parking function but $\piR^i=(u_1, \dots, u_{i-1}, u_i+1, u_{i+1}, \dots, u_l, \pi_{l+1}, \dots, \pi_m)$ is not for any $1 \leq i\leq l$. By (\ref{pigeon}), this could only happen when $\#\{k: \pi_k \leq u_i\}=m-n+u_i$ for all $1\leq i\leq l$. We claim that none of the subsequent $m-l$ cars can have preference $u_1, \dots, u_l$. Suppose otherwise and there is a later car with preference $u_i$. Such a car would necessarily park in spots $u_i+1, \dots, n$ for $\piR$, and consequently it could change places with car $i$ in $\piR^i$, contradicting the statement that $\pi_i=u_i$ is allowed but $\pi^i_i=u_i+1$ is not allowed. Hence excluding the first $l$ cars, $\piR$ has exactly $m-n+u_1-1$ cars with value $\leq u_1-1$, exactly $u_2-u_1-1$ cars with value $\geq u_1+1$ and $\leq u_2-1$, $\dots$, exactly $u_l-u_{l-1}-1$ cars with value $\geq u_{l-1}+1$ and $\leq u_l-1$, and exactly $n-u_l$ cars with value $\geq u_l+1$.

Let $\alphaR_1$ be the subsequence of $(\pi_{l+1}, \dots, \pi_m)$ with value $\leq u_1-1$, $\alphaR_2'$ be the subsequence with value $\geq u_1+1$ and $\leq u_2-1$, $\dots$, $\alphaR_l'$ be the subsequence with value $\geq u_{l-1}+1$ and $\leq u_l-1$, and $\alphaR_{l+1}'$ be the subsequence with value $\geq u_l+1$. Construct $\alphaR_2=\alphaR_2'-(u_1, \dots, u_1), \dots, \alphaR_{l+1}=\alphaR_{l+1}'-(u_l, \dots, u_l)$. It is clear from the above reasoning that $\alphaR_1 \in \PF(m-n+u_1-1, u_1-1), \alphaR_2 \in \PF(u_2-u_1-1, u_2-u_1-1), \dots, \alphaR_l \in \PF(u_l-u_{l-1}-1, u_l-u_{l-1}-1)$, and $\alphaR_{l+1} \in \PF(n-u_l, n-u_l)$. By Definition \ref{shuffle}, $(\pi_{l+1}, \dots, \pi_m) \in \MS(m-n+u_1-1, u_1-1, u_2-u_1-1, \dots, u_l-u_{l-1}-1, n-u_l)$.

``$\Longleftarrow$'' We first show that $\piR=(u_1, \dots, u_l, \pi_{l+1}, \dots, \pi_m)$ is a parking function. This is clear, since from Definition \ref{shuffle}, $\piR$ can be decomposed into $2l+1$ parts: a length $m-n+u_1-1$ subsequence $\alphaR_1$ with entries $\leq u_1-1$, one entry $u_1$, a length $u_2-u_1-1$ subsequence $\alphaR_2'$ with entries $\geq u_1+1$ and $\leq u_2-1$, one entry $u_2$, $\dots$, a length $u_l-u_{l-1}-1$ subsequence $\alphaR_l'$ with entries $\geq u_{l-1}+1$ and $\leq u_l-1$, one entry $u_l$, and a length $n-u_l$ subsequence $\alphaR_{l+1}'$ with entries $\geq u_l+1$. Moreover, $\alphaR_1, \alphaR_2=\alphaR_2'-(u_1, \dots, u_1), \dots, \alphaR_{l+1}=\alphaR_{l+1}'-(u_l, \dots, u_l)$ are $l+1$ parking functions.

Next we show that $\piR^i=(u_1, \dots, u_{i-1}, u_i+1, u_{i+1}, \dots, u_l, \pi_{l+1}, \dots, \pi_m)$ is not a parking function for any $1 \leq i\leq l$. But this is immediate since the only entries of $\piR^i$ that are bounded above by $u_i$ are those from $\alphaR_1, \alphaR_2', \dots, \alphaR_i'$ and $u_1, \dots, u_{i-1}$,
\begin{align}
\#\{k: \pi^i_k \leq u_i\}&=(m-n+u_1-1)+(u_2-u_1-1)+\cdots+(u_i-u_{i-1}-1)+i-1\notag \\
&=m-n+u_i-1<m-n+u_i,
\end{align}
a contradiction.

Combining, we have $A_{\pi_{l+1}, \dots, \pi_m}=[\uR]$.
\end{proof}

\begin{theorem}\label{component}
Take $1\leq l \leq m$ any integer. Let $\vR=(v_1, \dots, v_l) \in [n]^l$ be in increasing order. The number of parking functions $\piR\in \PF(m, n)$ with $\pi_1=v_1, \dots, \pi_l=v_l$ is
\begin{equation}
(n-m+1) \sum_{\sR \in S_l(\vR)} \binom{m-l}{\sR} (s_1+1+n-m)^{s_1-1}\prod_{i=2}^{l+1} (s_i+1)^{s_i-1},
\end{equation}
where
\begin{equation}
S_l(\vR)=\left\{\sR=(s_1, \dots, s_{l+1}) \in \mathbb{N}^{l+1} \left|\right. \substack{s_1+\cdots +s_i \geq m-n+v_i-i \hspace{.1cm} \forall i\in [l] \\ s_1+\cdots+s_{l+1}=m-l}\right\}.
\end{equation}
Note that this quantity stays constant if all $v_i \leq n-m+i$ and decreases as each $v_i$ increases past $n-m+i$ as there are fewer resulting summands.
\end{theorem}

\begin{proof}
If $\pi_i=v_i$ for $1\leq i\leq l$, then $A_{\pi_{l+1}, \dots, \pi_m}=[\uR]$ where $u_i\geq \max(v_i, n-m+i)$. Thus from Theorem \ref{main1}, the number of parking functions with $\pi_1=v_1, \dots, \pi_l=v_l$ is
\begin{align}
&\sum_{u_i=\max(v_i, n-m+i) \hspace{.1cm} \forall i\in [l]}^{n-l+i} \binom{m-l}{\sR} |\PF(m-n+u_1-1, u_1-1)| \cdot \notag \\
&\hspace{3cm} \cdot \prod_{i=2}^l |\PF(u_i-u_{i-1}-1, u_i-u_{i-1}-1)||\PF(n-u_l, n-u_l)| \notag \\
&=\sum_{u_i=\max(v_i, n-m+i) \hspace{.1cm} \forall i\in [l]}^{n-l+i} \binom{m-l}{\sR} (n-m+1) u_1^{m-n+u_1-2} \prod_{i=2}^l (u_i-u_{i-1})^{u_i-u_{i-1}-2}(n-u_l+1)^{n-u_l-1} \notag \\
&=(n-m+1) \sum_{\sR \in S_l(\vR)} \binom{m-l}{\sR} (s_1+1+n-m)^{s_1-1}\prod_{i=2}^{l+1} (s_i+1)^{s_i-1}, \label{withk}
\end{align}
where $\sR=(m-n+u_1-1, u_2-u_1-1, \dots, u_l-u_{l-1}-1, n-u_l)$.
\end{proof}

For the special case $l=0$ and $\vR=()$ (where no parking preferences are specified), we recover the total number of parking functions $|\PF(m, n)|=(n-m+1)(n+1)^{m-1}$. We describe an alternative characterization of this number in the following.

\begin{proposition}\label{un2}
The number of parking functions $|\PF(m, n)|$ satisfies
\begin{equation}
|\PF(m, n)|=\sum_{\sR \models m} \binom{m}{\sR} \prod_{i=1}^{n-m+1} (s_i+1)^{s_i-1},
\end{equation}
where $\sR=(s_1, \dots, s_{n-m+1})$ is a composition of $m$.
\end{proposition}

\begin{proof}
For a parking function $\piR\in \PF(m, n)$, there are $n-m$ parking spots that are never attempted by any car. Let $k_i(\piR)$ for $i=1, \dots, n-m$ represent these spots, so that $0:=k_0<k_1<\cdots<k_{n-m}<k_{n-m+1}:=n+1$. This separates $\piR$ into $n-m+1$ disjoint non-interacting segments (some segments might be empty), with each segment a classical parking function of length $(k_{i}-k_{i-1}-1)$ after translation. We have
\begin{align}
&|\PF(m, n)|=\sum_{k} \prod_{i=1}^{n-m+1} (k_{i}-k_{i-1})^{k_{i}-k_{i-1}-2} \binom{m}{k_1-k_0-1, \dots, k_{n-m+1}-k_{n-m}-1} \notag \\
&=\sum_{\sR \models m} \binom{m}{s_1, \dots, s_{n-m+1}} \prod_{i=1}^{n-m+1} (s_i+1)^{s_i-1},
\end{align}
where $\sR=(k_1-k_0-1, \dots, k_{n-m+1}-k_{n-m}-1)$ and $\sum_{i=1}^{n-m+1} s_i=m$.
\end{proof}

Building upon Theorem \ref{component} and Proposition \ref{un2}, we specialize to the case that the specified parking preferences of the first $l$ cars consist of a contiguous block.

\begin{proposition}\label{un}
Take $1\leq l \leq m$ any integer. Let $1\leq k\leq n-l+1$. The number of parking functions $\piR\in \PF(m, n)$ with $\pi_1=k, \dots, \pi_l=k+l-1$ is
\begin{equation}
(n-m+1) \sum_{s=0}^{\min(n-k-l+1, m-l)}\binom{m-l}{s} (n-s+1-l)^{m-s-l-1} l (s+l)^{s-1}.
\end{equation}
Note that this quantity stays constant for $k\leq n-m+1$ and decreases as $k$ increases past $n-m+1$ as there are fewer resulting summands.
\end{proposition}

\begin{proof}
We take $v_i=k+i-1$ for $1\leq i\leq l$ in Theorem \ref{component} and extract $s_1$ from the multinomial coefficient $\binom{m-l}{\sR}$:
\begin{multline}
(n-m+1) \sum_{s_1=\max(0, m-n+k-1)}^{m-l} \binom{m-l}{s_1} (s_1+1+n-m)^{s_1-1} \cdot \\ \cdot \sum_{(s_2, \dots, s_{l+1}) \models m-l-s_1} \binom{m-l-s_1}{s_2, \dots, s_{l+1}} \prod_{i=2}^{l+1} (s_i+1)^{s_i-1}.
\end{multline}
Using Proposition \ref{un2} and simplifying this becomes
\begin{align}
&(n-m+1) \sum_{s_1=\max(0, m-n+k-1)}^{m-l} \binom{m-l}{s_1} (s_1+1+n-m)^{s_1-1} |\PF(m-l-s_1, m-1-s_1)| \notag \\
&=(n-m+1) \sum_{s_1=\max(0, m-n+k-1)}^{m-l} \binom{m-l}{s_1} (s_1+1+n-m)^{s_1-1} l (m-s_1)^{m-l-s_1-1} \notag \\
&=(n-m+1) \sum_{s=0}^{\min(n-k-l+1, m-l)} \binom{m-l}{s} (n-s+1-l)^{m-s-l-1} l (s+l)^{s-1},
\end{align}
where the last equality is a change of variables $s=m-l-s_1$.
\end{proof}

Summing over all possible contiguous blocks that the first $l$ cars may occupy, the result simplifies nicely.

\begin{proposition}\label{decomposition}
Take $1\leq l \leq m$ any integer. Then
\begin{equation}
\sum_{k=1}^{n-l+1} \#\{\piR\in \PF(m, n): \pi_1=k, \dots, \pi_l=k+l-1\}=(n-m+1)(n+1)^{m-l}.
\end{equation}
\end{proposition}

\begin{proof}
The proof relies on an extension of Pollak's circle argument \cite{Pollak}. Add an additional space $n+1$, and arrange the spaces in a circle. Allow $n+1$ also as a preferred space. We first select a contiguous block of length $l$ for the first $l$ cars, which can be done in $n+1$ ways. Then for the remaining $m-l$ cars, there are $(n+1)^{m-l}$ possible preference sequences. Note that $\piR$ is a parking function if and only if the spot $n+1$ is left open. For $j \in \mathbb{Z}/(n+1)\mathbb{Z}$, the preference sequence $\pi+j(1, \dots, 1)$ (modulo $n+1$) gives an assignment whose missing spaces are the rotations by $j$ of the missing spaces for the assignment of
$\piR$. Since there are $n-m+1$ missing spaces for the assignment of any preference sequence, any preference sequence $\piR$ has $n-m+1$ rotations which are parking functions. Therefore
\begin{align}
\sum_{k=1}^{n-l+1} \#\{\piR\in \PF(m, n): \pi_1=k, \dots, \pi_l=k+l-1\}&=
\frac{n-m+1}{n+1}(n+1)(n+1)^{m-l}\notag\\
&=(n-m+1)(n+1)^{m-l}.
\end{align}
\end{proof}

For the special case $l=1$, Proposition \ref{decomposition} reduces to the decomposition of parking functions $\PF(m, n)$ according to the parking preference of the first car $\pi_1$.

\subsection{Connections with Pitman-Stanley polytopes}
\label{completion}
Denote the set of $\uR$-parking functions by $\PF(\uR)$. The following propositions are direct consequences of the parking criterion (\ref{pigeon}) and are equivalent in nature. See the beginning of Section \ref{pfs} for more explanation.

\begin{proposition}\label{forward}
Take $1\leq l\leq m$ any integer. Let $\vR=(v_1, \dots, v_l) \in [n]^l$ be in increasing order. Then $\piR=(v_1, \dots, v_l, \pi_{l+1}, \dots, \pi_{m}) \in \PF(m, n)$ if and only if $(\pi_{l+1}, \dots, \pi_{m}) \in \PF(\uR)$, where the $u_i$'s are the largest $m-l$ numbers in $\{n-m+1, \dots, n\} \texttt{\char92} \{v_1, \dots, v_l\}$, arranged in increasing order.
\end{proposition}

\begin{proposition}\label{backward}
Let $\uR=(u_1, \dots, u_m)$ be a positive-integer-valued vector with $u_1 < \cdots < u_m$. Let $\vR=(v_1, \dots, v_l)=[u_1, u_m] \texttt{\char92} \{u_1, \dots, u_m\}$, arranged in increasing order. Then $\piR=(v_1, \dots, v_l, \pi_1, \dots, \pi_m) \in \PF(u_m-u_1+1, u_m)$ if and only if $(\pi_1, \dots, \pi_m) \in \PF(\uR)$.
\end{proposition}

Knowledge on $\PF(\uR)$ thus lends knowledge on $\PF(m, n)$ where the first $l$ cars have specified parking preferences, with $l$ depending on the gaps in $\uR$, and vice versa. In \cite{PS1}, Pitman and Stanley introduced an $m$-dimensional polytope $\Pi_m$ and related the number of $\uR$-parking functions to the volume polynomial of $\Pi_m$. Let $\xR=(x_1, \dots, x_m)$ with $x_i>0$ for all $i$. Let
\begin{equation}
\Pi_m(\xR)=\left\{\yR\in \mathbb{R}^m: y_i\geq 0 \text{ and } y_1+\cdots+y_i \leq x_1+\cdots+x_i, \hspace{.2cm} \forall i\in [m] \right\}.
\end{equation}
The $m$-dimensional volume $V_m(\xR)=\text{Vol}(\Pi_m(\xR))$ is a homogeneous polynomial of degree $m$ in the variables $x_1, \dots, x_m$, and is called the volume polynomial of the Pitman-Stanley polytope. The volume definition may be extended when some of the $x_i$'s equal zero for $2\leq i\leq m$. Trivially, we take $V_m(\xR)=0$ if $x_1=0$.

\begin{theorem}[adapted from Pitman and Stanley \cite{PS1}]\label{Pitman-Stanley}
Take $m\geq 1$ any integer. Let $\uR=(u_1, \dots, u_m) \in \mathbb{N}^m$ with $u_1 \leq \cdots \leq u_m$. Let $\xR=\Delta \uR=(u_1, u_2-u_1, \dots, u_m-u_{m-1}) \in \mathbb{N}^m$. The number of $\uR$-parking functions $|\PF(\uR)|=m!V_m(\xR)$, where the volume polynomial
\begin{equation}
V_m(\xR)=\sum_{\kR \in K_m} \prod_{i=1}^m \frac{x_i^{k_i}}{k_i!}=\frac{1}{m!} \sum_{\kR \in K_m} \binom{m}{k_1, \dots, k_m} x_1^{k_1} \dots x_m^{k_m},
\end{equation}
and $K_m$ is the set of balanced vectors of length $m$, i.e.
\begin{equation}
K_m=\{\kR \in \mathbb{N}^m: k_1+\cdots+k_i \geq i, \hspace{.2cm} \forall i\in [m-1] \text{ and } k_1+\cdots+k_m=m \}.
\end{equation}
\end{theorem}

Though the index set and summation formula in Theorem \ref{Pitman-Stanley} resemble those of Theorem \ref{component}, we will show via an example that they are not parallel interpretations for parking functions, but rather complementary to each other.

\begin{example}
Take $m=4$, $n=5$, $\uR=(2, 5)$, $\vR=(3, 4)$, and $\xR=\Delta \uR=(2, 3)$. Then by Propositions \ref{forward} and \ref{backward}, $(v_1, v_2, \pi_1, \pi_2) \in \PF(4, 5)$ and $(\pi_1, \pi_2) \in \PF(\uR)$ both satisfy
\begin{align}
&(\pi_1, \pi_2) \in A:= \{(1, 1), (1, 2), (1, 3), (1, 4), (1, 5), (2, 1), \notag \\
&\hspace{2cm} (2, 2), (2, 3), (2, 4), (2, 5), (3, 1), (3, 2), (4, 1), (4, 2), (5, 1), (5, 2)\}.
\end{align}
From Theorem \ref{component},
\begin{equation}
|A|=2 \left( \binom{2}{1, 1, 0}3^0 2^0 1^{-1}+\binom{2}{1, 0, 1}3^0 1^{-1} 2^0+\binom{2}{2, 0, 0}4^1 1^{-1} 1^{-1}\right)=2(2+2+4)=16.
\end{equation}
From Theorem \ref{Pitman-Stanley},
\begin{equation}
|A|=\binom{2}{1, 1}2^1 3^1+\binom{2}{2, 0}2^2 3^0=12+4=16.
\end{equation}
We see that neither of the compositions of $|A|$ refines the other.
\end{example}

\section{Properties of random parking functions}
\label{random}
In this section we use the multi-shuffle construction introduced in Section \ref{pfs} to investigate various properties of a parking function chosen uniformly at random from $\PF(m, n)$. Sections \ref{mixed} through \ref{cov3} discuss the generic situation $m \lesssim n$, with Section \ref{mixed} focusing on mixed moments of multiple coordinates and Sections \ref{cov1} through \ref{cov3} focusing on covariances. Section \ref{special} discusses the special situation $m=n$. We will write our results in terms of coordinates $\pi_1, \dots, \pi_l$ of parking functions, where $1\leq l\leq m$ is any integer, and unattempted parking spots, which we denote by $k_i(\piR)$ for $i=1, \dots, n-m$. The parking coordinates satisfy permutation symmetry while the unattempted parking spots do not, so the statements in this section may be interpreted for any coordinates but are specific to the unattempted spots.

\subsection{Mixed moments of multiple coordinates}\label{mixed}
We begin with an asymptotic result for the mixed moments of two coordinates.
\begin{theorem}\label{mean}
Take $p,q \geq 1$ any integer. Take $m$ and $n$ large with $m=cn$ for some $0<c<1$. For parking function $\piR$ chosen uniformly at random from $\PF(m, n)$, we have
\begin{equation}\label{moment1}
\ER(\pi_1^p)=\frac{n^p}{p+1} \left(1+\frac{1}{n}\left(\frac{p+1}{2}-\frac{cp}{1-c}\right)+O\left(\frac{1}{n^2}\right)\right),
\end{equation}
and
\begin{equation}\label{moment2}
\ER(\pi_1^p \pi_2^q)=\frac{n^{p+q}}{(p+1)(q+1)} \left(1+\frac{1}{n}\left(\frac{p+q+2}{2}-\frac{c(p+q)}{1-c}\right)+O\left(\frac{1}{n^2}\right)\right).
\end{equation}
\end{theorem}

The proof of Theorem \ref{mean} will utilize the following lemma.

\begin{lemma}\label{technical_lemma}
Take $l \geq 1$ any integer and $n$ large. For $1\leq i\leq l$, take $p_i \geq 1$ any integer and $a_i \sim n$ with $a_1<\cdots<a_l$. Then
\begin{align}\label{left}
\sum_{\substack{\#\{i: \hspace{.05cm} \pi_i \leq a_k\} \geq k \\ \forall k \in [l]}} \prod_{i=1}^l \pi_i^{p_i}
=\frac{a_l^{\sum_{i=1}^l p_i+l}}{\prod_{i=1}^l (p_i+1)} \left(1+\frac{1}{n}\left(\frac{\sum_{i=1}^l p_i+l}{2}\right)+O\left(\frac{1}{n^2}\right)\right).
\end{align}
\end{lemma}

\begin{proof}
Notice that the left side of (\ref{left}) may be alternatively computed in stages.

\vspace{.1cm}

Stage 1: We sum up $\prod_{i=1}^l \pi_i^{p_i}$, where the $\pi_i$'s all range from $1$ to $a_l$.

Stage 2: We subtract the sum of $\prod_{i=1}^l \pi_i^{p_i}$, where the $\pi_i$'s all range from $a_1+1$ to $a_l$ (so none of the $\pi_i$'s $\leq a_1$).

Stage 3: We subtract the sum of $\prod_{i=1}^l \pi_i^{p_i}$, where one of the $\pi_i$'s ranges from $1$ to $a_1$ while the others all range from $a_2+1$ to $a_l$ (so only one of the $\pi_i$'s $\leq a_2$).

$\vdots$

Stage $l$: We subtract the sum of $\prod_{i=1}^l \pi_i^{p_i}$, where one of the $\pi_i$'s ranges from $1$ to $a_1$, one ranges from $1$ to $a_2$, $\dots$, one ranges from $1$ to $a_{l-2}$, while the two remaining $\pi_i$'s both range from $a_{l-1}+1$ to $a_l$ (so only $l-2$ of the $\pi_i$'s $\leq a_{l-1}$).

\vspace{.1cm}

\noindent For illustration, we perform this alternative procedure when $l=3$.
\begin{align}
&\sum_{\pi_1=1}^{a_3} \sum_{\pi_2=1}^{a_3} \sum_{\pi_3=1}^{a_3} \pi_1^{p_1}\pi_2^{p_2}\pi_3^{p_3}-\sum_{\pi_1=a_1+1}^{a_3} \sum_{\pi_2=a_1+1}^{a_3} \sum_{\pi_3=a_1+1}^{a_3} \pi_1^{p_1}\pi_2^{p_2}\pi_3^{p_3}\notag \\
&-\left(\sum_{\pi_1=1}^{a_1} \pi_1^{p_1} \sum_{\pi_2=a_2+1}^{a_3} \sum_{\pi_3=a_2+1}^{a_3} \pi_2^{p_2} \pi_3^{p_3}+\sum_{\pi_2=1}^{a_1} \pi_2^{p_2} \sum_{\pi_1=a_2+1}^{a_3} \sum_{\pi_3=a_2+1}^{a_3} \pi_1^{p_1} \pi_3^{p_3}\right.\notag \\
&\hspace{5cm} \left.+\sum_{\pi_3=1}^{a_1} \pi_3^{p_3} \sum_{\pi_1=a_2+1}^{a_3} \sum_{\pi_2=a_2+1}^{a_3} \pi_1^{p_1} \pi_2^{p_2}\right).
\end{align}
Since $a_i \sim n$, the sums subtracted in Stages $2$ through $l$ are all of lower order than the sum in Stage $1$. The conclusion then follows from standard asymptotic analysis on the leading order term.
\end{proof}

\begin{proof}[Proof of Theorem \ref{mean}]
We convert the parking preferences of the first two cars to an equivalent increasing order:
\begin{multline}\label{begin}
\sum_{j=1}^n \sum_{k=1}^n j^p k^q \#\{\piR\in \PF(m, n): \pi_1=j, \pi_2=k\}=\sum_{j=1}^{n-1} j^{p+q} \#\{\piR\in \PF(m, n): \pi_1=j, \pi_2=j+1\}\\
+\sum_{j=1}^{n-1} \sum_{k=j+1}^n (j^p k^q+j^qk^p) \#\{\piR\in \PF(m, n): \pi_1=j, \pi_2=k\}.
\end{multline}

By Theorem \ref{component}, the second term of (\ref{begin}) is
\begin{align}\label{mid}
&(n-m+1) \sum_{j=1}^{n-1} \sum_{k=j+1}^n (j^p k^q+j^qk^p) \sum_{s_1=\max(0, m-n+j-1)}^{m-2} \sum_{s_2=\max(0, m-n+k-2-s_1)}^{m-2-s_1} \binom{m-2}{s_1, s_2, m-2-s_1-s_2} \cdot \notag \\
&\hspace{1.5cm} \cdot (s_1+1+n-m)^{s_1-1} (s_2+1)^{s_2-1} (m-2-s_1-s_2+1)^{m-2-s_1-s_2-1}\notag \\
&=(n-m+1) \sum_{s_1=0}^{m-2} \sum_{s_2=0}^{m-2-s_1} \binom{m-2}{s_1, s_2, m-2-s_1-s_2} (s_1+1+n-m)^{s_1-1} (s_2+1)^{s_2-1} \cdot \notag \\ &\hspace{1.5cm} \cdot (m-2-s_1-s_2+1)^{m-2-s_1-s_2-1} \sum_{j=1}^{n-m+1+s_1} \sum_{k=j+1}^{n-m+2+s_1+s_2} (j^p k^q+j^qk^p).
\end{align}
We make a change of variables: $s=s_2$ and $t=m-2-s_1-s_2$. Then (\ref{mid}) becomes
\begin{align}\label{middle2}
&(n-m+1) \sum_{s=0}^{m-2} \sum_{t=0}^{m-2-s} \binom{m-2}{s, t, m-2-s-t} (n-1-s-t)^{m-3-s-t} \cdot \notag \\
&\hspace{1.5cm} \cdot (s+1)^{s-1}(t+1)^{t-1} \sum_{j=1}^{n-1-s-t} \sum_{k=j+1}^{n-t} (j^p k^q+j^qk^p).
\end{align}
Similarly, by Proposition \ref{un}, the first term of (\ref{begin}) is
\begin{align}\label{simple}
&(n-m+1) \sum_{j=1}^{n-1} j^{p+q} \sum_{s_1=\max(0, m-n+j-1)}^{m-2} \sum_{s_2=0}^{m-2-s_1} \binom{m-2}{s_1, s_2, m-2-s_1-s_2} \cdot \notag \\
&\hspace{1.5cm} \cdot (s_1+1+n-m)^{s_1-1} (s_2+1)^{s_2-1} (m-2-s_1-s_2+1)^{m-2-s_1-s_2-1}\notag \\
&=(n-m+1) \sum_{s_1=0}^{m-2} \sum_{s_2=0}^{m-2-s_1} \binom{m-2}{s_1, s_2, m-2-s_1-s_2} (s_1+1+n-m)^{s_1-1} (s_2+1)^{s_2-1} \cdot \notag \\ &\hspace{1.5cm} \cdot (m-2-s_1-s_2+1)^{m-2-s_1-s_2-1} \sum_{j=1}^{n-m+1+s_1} j^{p+q}.
\end{align}
We make a change of variables: $s=s_2$ and $t=m-2-s_1-s_2$. Then (\ref{simple}) becomes
\begin{align}\label{middle}
&(n-m+1) \sum_{s=0}^{m-2} \sum_{t=0}^{m-2-s} \binom{m-2}{s, t, m-2-s-t} (n-1-s-t)^{m-3-s-t} \cdot \notag \\
&\hspace{1.5cm} \cdot (s+1)^{s-1}(t+1)^{t-1} \sum_{j=1}^{n-1-s-t} j^{p+q}.
\end{align}

Using Lemma \ref{technical_lemma}, for $p, q \geq 1$, (\ref{middle2})+(\ref{middle}) is asymptotically
\begin{align}
&\label{before}\frac{n-m+1}{(p+1)(q+1)} \sum_{s=0}^{m-2} \sum_{t=0}^{m-2-s} \frac{m^{s+t}}{s! t!} n^{m-s-t+p+q-1} e^{-c(s+t+1)} (s+1)^{s-1}(t+1)^{t-1} \cdot \notag \\
&\cdot \left(1-\frac{(s+t)(s+t+3)}{2cn}+\frac{(s+t+1)(s+t+3)}{n}-\frac{t(p+q+2)}{n}\right.\notag \\
& \hspace{6cm}\left.-\frac{c(s+t+1)^2}{2n}+\frac{p+q+2}{2n}+O(n^{-2})\right).
\end{align}
The tree function $F(z) = \sum_{s=0}^\infty \frac{z^s}{s!}(s+1)^{s-1}$ is related to the Lambert $W$ function via $F(z)=-W(-z)/z$, and
satisfies $F(ce^{-c}) = e^c$. By the chain rule its first and second derivatives therefore satisfy
\begin{align}
F'(ce^{-c}) = \frac{e^{2c}}{1-c}, \hspace{1cm} F''(ce^{-c}) = \frac{3-2c}{(1-c)^3}e^{3c}.
\end{align}
We recognize that (\ref{before}) is in the form of a Cauchy product, and converges to
\begin{align}
&\frac{n-m+1}{(p+1)(q+1)} n^{m+p+q-1}e^{-c} \sum_{s=0}^{\infty} \sum_{t=0}^{\infty}\frac{(ce^{-c})^{s+t}}{s! t!} (s+1)^{s-1} (t+1)^{t-1} \cdot \notag \\
& \hspace{2cm} \cdot \left(1+\frac{1}{n} (A+Bs+Ct+Ds^2+Et^2+Fst)+O(n^{-2})\right),
\end{align}
where
\begin{equation*}
A=-\frac{c}{2}+3+\frac{p+q+2}{2}, \hspace{.2cm} B=-c-\frac{3}{2c}+4, \hspace{.2cm} C=-c-\frac{3}{2c}-p-q+2,
\end{equation*}
\begin{equation}
D=-\frac{c}{2}-\frac{1}{2c}+1, \hspace{.2cm} E=-\frac{c}{2}-\frac{1}{2c}+1, \hspace{.2cm} F=-c-\frac{1}{c}+2.
\end{equation}
Using $F(z)$ this can be written as (with $z=ce^{-c}$):
\begin{align}
&\frac{n-m+1}{(p+1)(q+1)} n^{m+p+q-1} \cdot \notag \\
&\cdot \left[F(z)+\frac1n\left(AF(z)+(B+C)zF'(z)+(D+E)(z^2F''(z)+zF'(z))+Fz^2 F'(z) \frac{F'(z)}{F(z)}\right)+O(n^{-2})\right].
\end{align}
Dividing by $|\PF(m,n)|=(n-m+1)(n+1)^{m-1}$ and simplifying we get
\begin{equation}\label{contri2}
\frac{n^{p+q}}{(p+1)(q+1)} \left(1+\frac{1}{n}\left(\frac{p+q+2}{2}-\frac{c(p+q)}{1-c}\right)+O\left(\frac{1}{n^2}\right)\right)
\end{equation}
for the generic $(p, q)$-th mixed moment.

For the special case $p\geq 1$ and $q=0$, a similar asymptotic calculation gives the $p$-th moment as
\begin{equation}\label{contri2-special}
\frac{n^p}{p+1} \left(1+\frac{1}{n}\left(\frac{p+1}{2}-\frac{cp}{1-c}\right)+O\left(\frac{1}{n^2}\right)\right).
\end{equation}
\end{proof}

Extending the asymptotic expansion approach in the proof of Theorem \ref{mean}, we have the following more general result.

\begin{theorem}\label{general-mean}
Take $l\geq 1$ any integer. For $1\leq i\leq l$, take $p_i \geq 1$ any integer. Take $m$ and $n$ large with $m=cn$ for some $0<c<1$. For parking function $\piR$ chosen uniformly at random from $\PF(m, n)$, we have
\begin{equation}\label{general-moment}
\ER(\prod_{i=1}^l \pi_i^{p_i})=\frac{n^{\sum_{i=1}^l p_i}}{\prod_{i=1}^l (p_i+1)} \left(1+\frac{1}{n}\left(\frac{\sum_{i=1}^l p_i+l}{2}-\frac{c\sum_{i=1}^l p_i}{1-c}\right)+O\left(\frac{1}{n^2}\right)\right).
\end{equation}
\end{theorem}

\begin{proof}
We will not include all technical details as in the $l=1, 2$ case, but point to some key facts. As in the Proof of Theorem \ref{mean}, using Theorem \ref{component} and Proposition \ref{un} and interchanging the order of summation, we have
\begin{align}\label{general-asy}
&\sum \left(\prod_{i=1}^l \pi_i^{p_i}\right) \{\piR\in \PF(m, n): \pi_i \text{ specified } \forall i \in [l]\} \notag \\
&=(n-m+1) \sum_{s_1=0}^{m-l} \cdots \sum_{s_l=0}^{m-l-s_1-\cdots-s_{l-1}} \binom{m-l}{s_1, \dots, s_l, m-l-s_1-\cdots-s_l} \cdot \notag \\
&\cdot (n-l+1-s_1-\cdots-s_l)^{m-l-1-s_1-\cdots-s_l} \prod_{i=1}^l (s_i+1)^{s_i-1} \left[\sum_{\substack{\#\{i: \hspace{.05cm} \pi_i \leq n-l+k-\sum_{j=k}^l s_j\} \geq k \\ \forall k \in [l]}} \prod_{i=1}^l \pi_i^{p_i}\right].
\end{align}
By Lemma \ref{technical_lemma}, for $p_i\geq 1$, (\ref{general-asy}) is asymptotically
\begin{align}
&\frac{n-m+1}{\prod_{i=1}^l (p_i+1)} \sum_{s_1=0}^{m-l} \cdots \sum_{s_l=0}^{m-l-s_1-\cdots-s_{l-1}} \frac{m^{\sum_{i=1}^l s_i}}{\prod_{i=1}^l s_i!} n^{m-1+\sum_{i=1}^l (p_i-s_i)} e^{-c\left(l-1+\sum_{i=1}^l s_i\right)} \prod_{i=1}^l (s_i+1)^{s_i-1} \cdot \notag \\
&\cdot \left(1-\frac{\left(\sum_{i=1}^l s_i\right)\left(2l-1+\sum_{i=1}^l s_i\right)}{2cn}+\frac{\left(l-1+\sum_{i=1}^l s_i\right)\left(l+1+\sum_{i=1}^l s_i\right)}{n}-\frac{s_l\left(\sum_{i=1}^l p_i+l\right)}{n}\right.\notag \\
& \hspace{1.5cm}\left.-\frac{c\left(l-1+\sum_{i=1}^l s_i\right)^2}{2n}+\frac{\sum_{i=1}^l p_i+l}{2n}+O\left(n^{-2}\right)\right). \label{generic}
\end{align}

Denote by $F(z) = \sum_{s=0}^\infty \frac{z^s}{s!}(s+1)^{s-1}$. An application of the tree function method shows that (\ref{generic}) converges to
\begin{align}
&\frac{n-m+1}{\prod_{i=1}^l (p_i+1)} n^{m-1+\sum_{i=1}^l p_i} \cdot \notag \\
&\cdot \Bigg[F(z)+\frac1n\left(AF(z)+\left(Bl-\sum_{i=1}^l p_i-l\right)zF'(z)+Cl(z^2F''(z)+zF'(z))+D\binom{l}{2}z^2 F'(z) \frac{F'(z)}{F(z)}\right)\notag \\
&\hspace{6cm}+O(n^{-2})\Bigg],
\end{align}
where
\begin{equation*}
A=-\frac{c(l-1)^2}{2}+(l^2-1)+\frac{\sum_{i=1}^l p_i+l}{2}, \hspace{.2cm} B=-c(l-1)-\frac{2l-1}{2c}+2l,
\end{equation*}
\begin{equation}
C=-\frac{c}{2}-\frac{1}{2c}+1, \hspace{.2cm} D=-c-\frac{1}{c}+2.
\end{equation}
Dividing by $|\PF(m,n)|=(n-m+1)(n+1)^{m-1}$ and simplifying we get
\begin{equation}
\frac{n^{\sum_{i=1}^l p_i}}{\prod_{i=1}^l (p_i+1)} \left(1+\frac{1}{n}\left(\frac{\sum_{i=1}^l p_i+l}{2}-\frac{c\sum_{i=1}^l p_i}{1-c}\right)+O\left(\frac{1}{n^2}\right)\right)
\end{equation}
for the generic mixed moment.
\end{proof}

Record the parking outcome of $\piR\in \PF(m, n)$ by $\tauR(\piR)=(\tau_1, \dots, \tau_m)$, where the $i$th car parks in spot $\tau_i$ with $1\leq \tau_i\leq n$. A similar asymptotic argument as in the proof of Theorems \ref{mean} and \ref{general-mean} leads to the following.

\begin{theorem}
Take $l\geq 1$ any integer. For $1\leq i\leq l$, take $p_i \geq 1$ any integer. Take $m$ and $n$ large with $m=cn$ for some $0<c<1$. For parking function $\piR$ chosen uniformly at random from $\PF(m, n)$, we have
\begin{equation}\label{general-moment}
\ER(\prod_{i=1}^l \tau_i^{p_i})=\frac{n^{\sum_{i=1}^l p_i}}{\prod_{i=1}^l (p_i+1)} \left(1+\frac{1}{n}\left(\frac{\sum_{i=1}^l p_i+l}{2}-\frac{c\sum_{i=1}^l p_i}{1-c}\right)+O\left(\frac{1}{n^2}\right)\right),
\end{equation}
where $\tauR$ is the parking outcome of $\piR$. In particular, for any finite $i$,
\begin{equation}
\ER(\tau_i^{p_i})=\frac{n^{p_i}}{p_i+1} \left(1+\frac{1}{n}\left(\frac{p_i+1}{2}-\frac{cp_i}{1-c}\right)+O\left(\frac{1}{n^2}\right)\right).
\end{equation}
\end{theorem}

We will now deduce all possible covariances of parking functions, between two coordinates, between a coordinate and an unattempted spot, and between two unattempted spots. As for the mixed moment calculations in Section \ref{mixed}, combinatorial consideration and asymptotic expansion will be the central ingredients in our derivations.

\subsection{Covariance between two coordinates}\label{cov1}
\begin{proposition}\label{two-coordinates}
Take $m$ and $n$ large with $m=cn$ for some $0<c<1$. For parking function $\piR$ chosen uniformly at random from $\PF(m, n)$, we have
\begin{equation}
\Var(\pi_1) \sim \frac{1}{12}n^2-\frac{c}{6(1-c)}n, \hspace{1cm} \Cov(\pi_1, \pi_2) \sim -\frac{1}{4(1-c)^2}.
\end{equation}
\end{proposition}

\begin{proof}
For $p=q=1$, performing asymptotic expansion as in the proof of Theorem \ref{mean} but keeping more lower order terms, we have
\begin{align*}
\sum_{j=1}^n \sum_{k=1}^n jk \#\{\piR\in \PF(m, n): \pi_1=j, \pi_2=k\}
\end{align*}
converges to
\begin{align*}
&\frac{n-m+1}{4} n^{m+1}e^{-c} \sum_{s=0}^{\infty} \sum_{t=0}^{\infty} \frac{(ce^{-c})^{s+t}}{s! t!} (s+1)^{s-1} (t+1)^{t-1} \cdot \notag \\
&\cdot \left(1+\left(A_1+A_2s+A_3t+A_4s^2+A_5t^2+A_6st\right)\frac{1}{n}+\left(B_1+B_2s+B_3t+B_4s^2+B_5t^2+B_6st+ \right.\right. \notag \\
&\left.+B_7s^2t+B_8st^2+B_9s^3+B_{10}t^3+B_{11}s^3t+B_{12}st^3+B_{13}s^2t^2+B_{14}s^4+B_{15}t^4)\frac{1}{n^2} +O(n^{-3}) \right),
\end{align*}
where
\begin{equation*}
A_1=-\frac{c}{2}+5, \hspace{.5cm} A_2+A_3=-2c-\frac{3}{c}+4,
\end{equation*}
\begin{equation*}
A_4+A_5=-c-\frac{1}{c}+2, \hspace{.5cm} A_6=-c-\frac{1}{c}+2,
\end{equation*}
\begin{equation*}
B_1=\frac{c^2}{8}-\frac{17c}{6}+9, \hspace{.5cm} B_2+B_3=c^2-\frac{13}{6c^2}-14c-\frac{15}{c}+\frac{45}{2},
\end{equation*}
\begin{equation}
B_4+B_5=\frac{3c^2}{2}+\frac{3}{4c^2}-12c-\frac{11}{c}+\frac{41}{2}, \hspace{.5cm} B_6=\frac{3c^2}{2}+\frac{3}{4c^2}-12c-\frac{11}{c}+\frac{37}{2},
\end{equation}
\begin{equation*}
B_7+B_8=3c^2+\frac{7}{2c^2}-14c-\frac{15}{c}+\frac{45}{2}, \hspace{.5cm} B_9+B_{10}=c^2+\frac{7}{6c^2}-\frac{14c}{3}-\frac{5}{c}+\frac{15}{2},
\end{equation*}
\begin{equation*}
B_{11}+B_{12}=c^2+\frac{1}{c^2}-4c-\frac{4}{c}+6, \hspace{.5cm} B_{13}=\frac{3c^2}{4}+\frac{3}{4c^2}-3c-\frac{3}{c}+\frac{9}{2},
\end{equation*}
\begin{equation*}
B_{14}+B_{15}=\frac{c^2}{4}+\frac{1}{4c^2}-c-\frac{1}{c}+\frac{3}{2}.
\end{equation*}
A more involved application of the tree function method then yields
\begin{equation}\label{last2}
\ER(\pi_1\pi_2) \sim \frac{n^2}{4}+\frac{(1-2c)n}{2(1-c)}+\frac{1-c+3c^2-2c^3}{2(1-c)^3}.
\end{equation}
The same approach also yields
\begin{equation}\label{last1}
\ER(\pi_1) \sim \frac{n}{2}+\frac{1-2c}{2(1-c)}+\frac{1+c-c^2}{2(1-c)^3n}.
\end{equation}
The claimed asymptotics are then immediate.
\end{proof}

\subsection{Covariance between a coordinate and an unattempted spot}\label{cov2}
Recall that for a parking function $\piR\in \PF(m, n)$, there are $n-m$ parking spots that are never attempted by any car. Let $k_i(\piR)$ for $i=1, \dots, n-m$ represent these spots, so that $0:=k_0<k_1<\cdots<k_{n-m}<k_{n-m+1}:=n+1$. Let
\begin{equation}
\PF(m, n; i, k)=\{\piR\in \PF(m, n): k_i(\piR)=k\},
\end{equation}
consisting of parking functions where the $i$th empty spot is fixed at $k$. The unattempted spot $k$ ranges from $i$ to $m+i$ and breaks up the parking function $\piR$ into two components $\alphaR$ and $\betaR$, with $\alphaR \in \PF(k-i, k-1)$ and $\betaR \in \PF(m-k+i, n-k)$, and $\piR$ a shuffle of the two. From the parking scheme, if $j<k$ and $\piR=(j, \pi_2, \dots, \pi_m) \in \PF(m, n; i, k)$, then $\piR'=(l, \pi_2, \dots, \pi_m) \in \PF(m, n; i, k)$ for all $1\leq l\leq j$, while if $j>k$ and $\piR=(j, \pi_2, \dots, \pi_m) \in \PF(m, n; i, k)$, then $\piR'=(l, \pi_2, \dots, \pi_m) \in \PF(m, n; i, k)$ for all $k+1\leq l\leq j$. This implies that given the last $m-1$ parking preferences, it is sufficient to identify the largest feasible first preference (if exists).

\begin{theorem}\label{main2}
$\piR=(j, \pi_2, \dots, \pi_m)$ is in $\PF(m, n; i, k)$ but $\piR'=(j+1, \pi_2, \dots, \pi_m)$ is not if and only if (1) $i\leq j\leq k-1$ and $(\pi_2, \dots, \pi_m)$ is a multi-shuffle of $\alphaR \in \PF(j-i, j-1)$, $\betaR \in \PF(k-j-1, k-j-1)$, and $\gammaR \in \PF(m-k+i, n-k)$; or (2) $j \geq n-m-i+k+1$ and $(\pi_2, \dots, \pi_m)$ is a multi-shuffle of $\alphaR \in \PF(k-i, k-1)$, $\betaR \in \PF(j-k-1-n+m+i, j-k-1)$, and $\gammaR \in \PF(n-j, n-j)$.
\end{theorem}

\begin{proof}
The proof builds upon Theorem \ref{main1}.

First suppose $j<k$. Then $(\pi_2, \dots, \pi_m)=(\delta_1, \dots, \delta_{k-i-1}, \gamma_1, \dots, \gamma_{m-k+i}):=(\deltaR, \gammaR)$, where $\deltaR$ consists of cars with preference $\leq k-1$ and $\gammaR$ consists of cars with preference $\geq k+1$. It is clear that $(\pi_1, \deltaR) \in \PF(k-i, k-1)$ and $\gammaR \in \PF(m-k+i, n-k)$. The statement of the theorem is equivalent to identifying $j$ so that $A_{\deltaR}=[j]$. From Theorem \ref{main1}, $j \geq i$ and $\deltaR$ is a shuffle of $\alphaR \in \PF(j-i, j-1)$ and $\betaR \in \PF(k-j-1, k-j-1)$.

Next suppose $j>k$. Then $(\pi_2, \dots, \pi_m)=(\alpha_1, \dots, \alpha_{k-i}, \delta_1, \dots, \delta_{m-k+i-1}):=(\alphaR, \deltaR)$, where $\alphaR$ consists of cars with preference $\leq k-1$ and $\deltaR$ consists of cars with preference $\geq k+1$. It is clear that $\alphaR \in \PF(k-i, k-1)$ and $(\pi_1, \deltaR) \in \PF(m-k+i, n-k)$. The statement of the theorem is equivalent to identifying $j$ so that $A_{\deltaR}=[j-k]$. From Theorem \ref{main1}, $j-k \geq n-m-i+1$ and $\deltaR$ is a shuffle of $\betaR \in \PF(j-k-1-n+m+i, j-k-1)$ and $\gammaR \in \PF(n-j, n-j)$.
\end{proof}

\begin{proposition}\label{jk_cov}
Take $1\leq i \leq n-m$ any integer. Take $i \leq k \leq m+i$ any integer.
For $j<k$, the number of parking functions $\piR\in \PF(m, n)$ with $\pi_1=j$ and $k_i=k$ is
\begin{equation}
\binom{m-1}{m-k+i} i(n-m-i+1) (n-k+1)^{m-k+i-1} \sum_{s=0}^{\min(k-i-1, k-j-1)} \binom{k-i-1}{s} (k-1-s)^{k-i-s-2} (s+1)^{s-1}.
\end{equation}
Note that this quantity stays constant for $j \leq i$ and decreases as $j$ increases past $i$ as there are fewer resulting summands.
For $j>k$, the number of parking functions $\piR\in \PF(m, n)$ with $\pi_1=j$ and $k_i=k$ is
\begin{equation}
\binom{m-1}{k-i} i k^{k-i-1} (n-m-i+1) \sum_{s=0}^{\min(m+i-k-1, n-j)} \binom{m-k+i-1}{s} (n-k-s)^{m+i-k-s-2} (s+1)^{s-1}.
\end{equation}
Note that this quantity stays constant for $j \leq n-m-i+k+1$ and decreases as $j$ increases past $n-m-i+k+1$ as there are fewer resulting summands.
\end{proposition}

\begin{proof}
If $\pi_1=j<k$, then the maximal $\pi_1$ consistent with $\pi_2, \dots, \pi_m$ and $k_i$ is some $l \geq \max(j, i)$ and $\leq k-1$. Thus from Theorem \ref{main2}, the number of parking functions with $\pi_1=j$ and $k_i=k$ is
\begin{align}
&\sum_{l=\max(j, i)}^{k-1} \binom{m-1}{l-i, k-l-1, m-k+i} |\PF(l-i, l-1)| \cdot \notag \\
& \hspace{5cm} \cdot |\PF(k-l-1, k-l-1)| |\PF(m-k+i, n-k)| \notag \\
&=\sum_{l=\max(j, i)}^{k-1} \binom{m-1}{l-i, k-l-1, m-k+i} i(n-m-i+1) l^{l-i-1} (k-l)^{k-l-2} (n-k+1)^{m-k+i-1} \notag \\
&=\binom{m-1}{m-k+i} i(n-m-i+1) (n-k+1)^{m-k+i-1} \cdot \notag \\
& \hspace{5cm} \cdot \sum_{s=0}^{\min(k-i-1, k-j-1)} \binom{k-i-1}{s} (k-1-s)^{k-i-s-2} (s+1)^{s-1},
\end{align}
where the last equality is a change of variables $s=k-l-1$.

If $\pi_1=j>k$, then the maximal $\pi_1$ consistent with $\pi_2, \dots, \pi_m$ and $k_i$ is some $l \geq \max(j, n-m-i+k+1)$. Thus from Theorem \ref{main2}, the number of parking functions with $\pi_1=j$ and $k_i=k$ is
\begin{align*}
&\sum_{l=\max(j, n-m-i+k+1)}^n \binom{m-1}{k-i, l-k-1-n+m+i, n-l} |\PF(k-i, k-1)| \cdot \notag \\
& \hspace{5cm} \cdot |\PF(l-k-1-n+m+i, l-k-1)| |\PF(n-l, n-l)| \notag \\
&=\sum_{l=\max(j, n-m-i+k+1)}^n \binom{m-1}{k-i, l-k-1-n+m+i, n-l} i k^{k-i-1} (n-m-i+1) \cdot \notag \\
& \hspace{5cm} \cdot (l-k)^{l-k-n+m+i-2} (n-l+1)^{n-l-1} \notag \\
&=\binom{m-1}{k-i} i k^{k-i-1} (n-m-i+1)  \cdot \notag \\
& \hspace{5cm} \cdot \sum_{s=0}^{\min(m+i-k-1, n-j)} \binom{m-k+i-1}{s} (n-k-s)^{m+i-k-s-2} (s+1)^{s-1},
\end{align*}
where the last equality is a change of variables $s=n-l$.
\end{proof}

\begin{proposition}\label{coordinate-spot}
Take $i \geq 1$ any integer. Take $m$ and $n$ large with $m=cn$ for some $0<c<1$. For parking function $\piR$ chosen uniformly at random from $\PF(m, n)$, we have
\begin{equation}
\Cov(\pi_1, k_i) \sim -\frac{i}{2(1-c)^2}.
\end{equation}
\end{proposition}

\begin{proof}
From Proposition \ref{jk_cov} and interchanging the order of summation, we have
\begin{align}\label{combined}
&\sum_{k=i}^{m+i} k \left(\sum_{j=1}^{k-1} j \#\{\piR\in \PF(m, n): \pi_1=j, k_i=k\}+\sum_{j=k+1}^{n} j \#\{\piR\in \PF(m, n): \pi_1=j, k_i=k\}\right) \notag \\
&=\frac{1}{2}i(n-m-i+1) \left[\sum_{s=0}^{m-1} \sum_{k=i+1+s}^{m+i} \binom{m-1}{m-k+i, s, k-i-s-1}k (n-k+1)^{m-k+i-1} \cdot\right. \notag \\
& \hspace{5cm} \cdot (k-1-s)^{k-i-s} (s+1)^{s-1} \left(1+\frac{1}{k-1-s}\right) \notag \\
&\left.+\sum_{s=0}^{m-1} \sum_{k=i}^{m+i-1-s} \binom{m-1}{k-i, s, m-k+i-1-s} k^{k-i} (n-k-s)^{m+i-k-s} (s+1)^{s-1}\left(1+\frac{2k+1}{n-k-s}\right)\right].
\end{align}
We make a change of variables: $t=k-i-s-1$ in the first sum and $t=k-i$ in the second sum. Then (\ref{combined}) becomes
\begin{align}\label{almost}
&\frac{1}{2}i(n-m-i+1) \sum_{s=0}^{m-1} \sum_{t=0}^{m-1-s} \binom{m-1}{s, t, m-1-s-t} (n-i-s-t)^{m-s-t-2} (s+1)^{s-1} \cdot \notag \\
& \hspace{.5cm} \cdot \left[(s+t+i+1)(t+i)^{t+1} \left(1+\frac{1}{t+i}\right)+(t+i)^{t} (n-i-s-t)^{2} \left(1+\frac{2(t+i)+1}{n-i-s-t}\right)\right] \notag \\
&= \frac{1}{2}i(n-m-i+1) n^{m} e^{-ic} \sum_{s=0}^{m-1} \sum_{t=0}^{m-1-s} \frac{(ce^{-c})^{s+t}}{s!t!} (s+1)^{s-1} (t+i)^{t} \cdot \notag \\
&\hspace{.5cm} \cdot \left(1-\frac{(s+t)(s+t+1)}{2cn}-\frac{c(s+t+i)^2}{2n}+\frac{(s+t+i)(s+t)}{n}+\frac{2(t+i)+1}{n}+O\left(n^{-2}\right)\right).
\end{align}

The generalized tree function $F_i(z)=\sum_{s=0}^\infty \frac{z^s}{s!} (s+i)^{s-1}$ is related to the tree function $F_1(z)$ via $F_i(z)=(F_1(z))^i/i$, and satisfies $F_i(ce^{-c})=e^{ic}/i$. Further, $G_i(z)=\sum_{s=0}^\infty \frac{z^s}{s!} (s+i)^{s}=(F_{i-1}(z))'$. By the chain rule the first and second derivatives of $F_1(z)$ and $G_i(z)$ therefore respectively satisfy
\begin{align*}
F_1'(ce^{-c}) = \frac{e^{2c}}{1-c}, \hspace{1cm} F_1''(ce^{-c}) = \frac{3-2c}{(1-c)^3}e^{3c},
\end{align*}
\begin{align}
G_i(ce^{-c}) = \frac{e^{ic}}{1-c}, \hspace{1cm} G_i'(ce^{-c}) = \frac{i+1-ic}{(1-c)^3}e^{(i+1)c},
\end{align}
\begin{align*}
G_i''(ce^{-c}) = \frac{(1-c)^2i^2+(1-c)(4-c)i+(4-c)}{(1-c)^5} e^{(i+2)c}.
\end{align*}
We recognize that (\ref{almost}) is in the form of a Cauchy product, and converges to
\begin{align*}
&\frac{1}{2}i(n-m-i+1) n^{m} e^{-ic} \sum_{s=0}^{\infty} \sum_{t=0}^{\infty}\frac{(ce^{-c})^{s+t}}{s! t!} (s+1)^{s-1} (t+i)^{t} \cdot \notag \\
& \hspace{2cm} \cdot \left(1+\frac{1}{n} (A+Bs+Ct+Ds^2+Et^2+Fst)+O(n^{-2})\right),
\end{align*}
where
\begin{equation*}
A=1+2i-\frac{ci^2}{2}, \hspace{.5cm} B=-\frac{1}{2c}+i-ci,
\end{equation*}
\begin{equation}
C=-\frac{1}{2c}+i-ci+2, \hspace{.5cm} D=1-\frac{1}{2c}-\frac{c}{2},
\end{equation}
\begin{equation*}
E=1-\frac{1}{2c}-\frac{c}{2}, \hspace{.5cm} F=2-\frac{1}{c}-c.
\end{equation*}
Using $F_1(z)$ and $G_i(z)$ this can be written as (with $z=ce^{-c}$):
\begin{align}
&\frac{1}{2}i(n-m-i+1) n^{m} e^{-ic} \left[F_1(z)G_i(z)+\frac1n\Big(AF_1(z)G_i(z)+BzF_1'(z)G_i(z)+CzF_1(z)G_i'(z)+\Big.\right. \notag \\ &\left.\Big.D(z^2F_1''(z)+zF_1'(z))G_i(z)+EF_1(z)(z^2G_i''(z)+zG_i'(z))+Fz^2 F_1'(z)G_i'(z)\Big)+O\left(\frac{1}{n^2}\right)\right].
\end{align}
Dividing by $|\PF(m,n)|=(n-m+1)(n+1)^{m-1}$ and simplifying we get
\begin{equation}
\ER(\pi_1 k_i) \sim \frac{in}{2(1-c)}-\frac{3ic}{2(1-c)^2}.
\end{equation}

The same approach also yields
\begin{equation}\label{ki}
\ER(k_i) \sim \frac{i}{1-c}-\frac{ic}{(1-c)^2 n}.
\end{equation}
Combining with Theorem \ref{mean}, the claimed asymptotics are then immediate.
\end{proof}

\subsection{Covariance between two unattempted spots}\label{cov3}

\begin{proposition}\label{two-spots}
Take $1\leq i<j$ any distinct integer pairs. Take $m$ and $n$ large with $m=cn$ for some $0<c<1$. For parking function $\piR$ chosen uniformly at random from $\PF(m, n)$, we have
\begin{equation}
\Var(k_i) \sim \frac{ic}{(1-c)^3}, \hspace{1cm} \Cov(k_i, k_j) \sim \frac{ic}{(1-c)^3}.
\end{equation}
\end{proposition}

\begin{proof}
Take $k_i(\piR)=k$ and $k_j(\piR)=l$. The unattempted spot $k$ ranges from $i$ to $m+i$ and the unattempted spot $l$ ranges from $k-i+j$ to $m+j$. The two unattempted spots break up the parking function $\piR$ into three components $\alphaR$, $\betaR$, and $\gammaR$, with $\alphaR \in \PF(k-i, k-1)$, $\betaR \in \PF(l-k-j+i, l-k-1)$, and $\gammaR \in \PF(m-l+j, n-l)$, and $\piR$ a multi-shuffle of the three. We have
\begin{align}\label{stretch}
&\sum_{k=i}^{m+i} k\sum_{l=k-i+j}^{m+j} l \#\{\piR\in \PF(m, n): k_i=k, k_j=l\}\notag \\
&=\sum_{k=i}^{m+i} k\sum_{l=k-i+j}^{m+j} l \binom{m}{k-i, l-k-j+i, m-l+j} i k^{k-i-1} \cdot \notag \\
&\hspace{2cm} \cdot (j-i) (l-k)^{l-k-j+i-1} (n-m-j+1) (n-l+1)^{m-l+j-1}.
\end{align}
We make a change of variables: $s=k-i$ and $t=l-k-j+i$. Then (\ref{stretch}) becomes
\begin{align}\label{lap}
&i(j-i)(n-m-j+1) \sum_{s=0}^m \sum_{t=0}^{m-s} \binom{m}{s, t, m-s-t} (s+i)^s (t+j-i)^{t-1} \cdot \notag \\
&\hspace{2cm} \cdot (j+s+t)(n-j-s-t+1)^{m-s-t-1} \notag \\
&= i(j-i)(n-m-j+1) n^{m-1} e^{-c(j-1)} \sum_{s=0}^m \sum_{t=0}^{m-s} \frac{(ce^{-c})^{s+t}}{s!t!}  \cdot \notag \\
&\hspace{2.5cm} \cdot \Big[(s+i)^{s+1} (t+j-i)^{t-1}+(s+i)^s (t+j-i)^t \Big] \left(1+O\left(n^{-1}\right)\right).
\end{align}

The generalized tree function $F_i(z)=\sum_{s=0}^\infty \frac{z^s}{s!} (s+i)^{s-1}$ is related to the tree function $F_1(z)$ via $F_i(z)=(F_1(z))^i/i$, and satisfies $F_i(ce^{-c})=e^{ic}/i$. Further, $G_i(z)=\sum_{s=0}^\infty \frac{z^s}{s!} (s+i)^{s}=(F_{i-1}(z))'$ and $H_i(z)=\sum_{s=0}^\infty \frac{z^s}{s!} (s+i)^{s+1}=(F_{i-2}(z))''$. By the chain rule $G_i(z)$ and $H_i(z)$ therefore respectively satisfy
\begin{align}
G_i(ce^{-c}) = \frac{e^{ic}}{1-c}, \hspace{1cm} H_i(ce^{-c}) = \frac{(1-c)i+c}{(1-c)^3}e^{ic}.
\end{align}
We recognize that (\ref{lap}) is in the form of a Cauchy product, and converges to
\begin{align*}
&i(j-i)(n-m-j+1) n^{m-1} e^{-c(j-1)} \cdot \notag \\
& \cdot \sum_{s=0}^{\infty} \sum_{t=0}^{\infty}\frac{(ce^{-c})^{s+t}}{s! t!} \Big[(s+i)^{s+1} (t+j-i)^{t-1}+(s+i)^s (t+j-i)^t \Big] \left(1+O\left(n^{-1}\right)\right).
\end{align*}
Using $F_i(z)$, $G_i(z)$, and $H_i(z)$ this can be written as (with $z=ce^{-c}$):
\begin{align}
&i(j-i)(n-m-j+1) n^{m-1} e^{-c(j-1)}  \bigg[H_i(z)F_{j-i}(z)+G_i(z)G_{j-i}(z)+O\left(\frac{1}{n}\right)\bigg].
\end{align}
Dividing by $|\PF(m,n)|=(n-m+1)(n+1)^{m-1}$ and simplifying we get
\begin{equation}
\ER(k_i k_j) \sim \frac{i(c+j-jc)}{(1-c)^3}.
\end{equation}

The same approach also yields
\begin{equation}
\ER(k_i^2) \sim \frac{i(c+i-ic)}{(1-c)^3}.
\end{equation}
Combining with (\ref{ki}), the claimed asymptotics are then immediate.
\end{proof}

\subsection{The special situation $m=n$}
\label{special}
The asymptotic moment calculations in Sections \ref{mixed}, \ref{cov1}, \ref{cov2}, and \ref{cov3} could be alternatively approached via Abel's multinomial theorem. Unlike the tree function method which fails for the case $m=n$ due to divergence, Abel's multinomial theorem applies broadly, whether in the generic case $m \lesssim n$ or in the special case $m=n$. However calculation-wise it is in general more cumbersome to apply Abel's multinomial theorem as compared with the tree function method, so we only use this alternative approach when $m=n$.

\begin{theorem}[Abel's multinomial theorem, derived from Pitman \cite{Pitman} and Riordan \cite{Riordan}]\label{Abel}
Let
\begin{equation}\label{b}
A_n(x_1, \dots, x_m; p_1, \dots, p_m)=\sum \binom{n}{\sR} \prod_{j=1}^m (x_j+s_j)^{s_j+p_j},
\end{equation}
where $\sR=(s_1, \dots, s_m)$ and $\sum_{i=1}^m s_i=n$.
Then
\begin{multline}\label{b1}
A_n(x_1, \dots, x_i, \dots, x_j, \dots, x_m; p_1, \dots, p_i, \dots, p_j, \dots, p_m)\\=A_n(x_1, \dots, x_j, \dots, x_i, \dots, x_m; p_1, \dots, p_j, \dots, p_i, \dots, p_m).
\end{multline}

\begin{multline}\label{b2}
A_n(x_1, \dots, x_m; p_1, \dots, p_m)\\=\sum_{i=1}^m A_{n-1}(x_1, \dots, x_{i-1}, x_i+1, x_{i+1}, \dots, x_m; p_1, \dots, p_{i-1}, p_i+1, p_{i+1}, \dots, p_m).
\end{multline}

\begin{equation}\label{b3}
A_n(x_1, \dots, x_m; p_1, \dots, p_m)=\sum_{s=0}^{n} \binom{n}{s}s!(x_1+s)A_{n-s}(x_1+s, x_2, \dots, x_m; p_1-1, p_2, \dots, p_m).
\end{equation}
Moreover, the following special instances hold via the basic recurrences listed above:
\begin{equation}\label{1}
A_n(x_1, \dots, x_m; -1, \dots, -1)=(x_1\cdots x_m)^{-1}(x_1+\cdots+x_m)(x_1+\cdots+x_m+n)^{n-1}.
\end{equation}

\begin{equation}\label{2}
A_n(x_1, \dots, x_m; -1, \dots, -1, 0)=(x_1\cdots x_m)^{-1}x_m(x_1+\cdots+x_m+n)^{n}.
\end{equation}
\end{theorem}

We recognize that in computing $\ER(\prod_{i=1}^l \pi_i^{p_i})$ in Theorem \ref{general-mean}, (\ref{general-asy}) is asymptotically
\begin{align}
&\frac{n-m+1}{\prod_{i=1}^l (p_i+1)} \left(A_{m-l}(n-m+1, \underbracket[0.5pt]{1, \dots, 1}_{l \hspace{.1cm} \text{1's}}; \sum_{i=1}^l p_i+l-1, \underbracket[0.5pt]{-1, \dots, -1}_{l \hspace{.1cm} \text{-1's}})\right.\notag \\
&+(l-1)\left(\sum_{i=1}^l p_i+l\right) A_{m-l}(n-m+1, \underbracket[0.5pt]{1, \dots, 1}_{l \hspace{.1cm} \text{1's}}; \sum_{i=1}^l p_i+l-2, 0, \underbracket[0.5pt]{-1, \dots, -1}_{l-1 \hspace{.1cm} \text{-1's}})\notag \\
&\left.+\frac{1}{2}\left(\sum_{i=1}^l p_i+l\right) A_{m-l}(n-m+1, \underbracket[0.5pt]{1, \dots, 1}_{l \hspace{.1cm} \text{1's}}; \sum_{i=1}^l p_i+l-2, \underbracket[0.5pt]{-1, \dots, -1}_{l \hspace{.1cm} \text{-1's}})\right).
\end{align}
This is a general formula that works for any $m$, $n$, and $l$. When $m=n$, taking $l=1, 2$, we have
\begin{equation}
\ER(\pi_1) \sim \frac{n}{2}-\frac{\sqrt{2\pi}}{4}n^{1/2}+\frac{5}{3}.
\end{equation}
\begin{equation}
\ER(\pi_1 \pi_2) \sim \frac{n^2}{4}-\frac{\sqrt{2\pi}}{4}n^{3/2}+2n.
\end{equation}
These asymptotic results are in sharp contrast with the case $m=cn$ for some $0<c<1$. As $c \rightarrow 1$, the correction terms in (\ref{last2}) (\ref{last1}) blow up, contributing to the different asymptotic orders between the generic situation $m \lesssim n$ and the special situation
$m=n$.

\section{Interval parking functions}
\label{ipf}
In this section we study a generalization of parking functions $\PF(m, n)$ in which the $i$th car is willing to park only in an interval $[a_i,b_i]\subseteq\{1,\dots,n\}$.  If all cars can successfully park then we say that the pair $(\aR,\bR)=((a_1,\dots,a_m), (b_1,\dots,b_m))$ is an \emph{interval parking function} with $m$ cars and $n$ spots, or $\IPF(m, n)$. If $b_i=n$ for all $i$, then we recover a parking function $\PF(m, n)$.

Let $\tauR(\cdot)$ denote the parking outcome of either a parking function or an interval parking function. The following propositions for $\IPF(m, n)$ generalize the corresponding results for the special case $\IPF(n, n)$ discussed in \cite{CDMY}.

\begin{proposition}\label{IPF-equivalence}
Let $\aR, \bR \in [n]^m$. Then
\begin{enumerate}
\item $\aR \in \PF(m, n)$ if and only if $(\aR,(n,\dots,n))\in\IPF(m, n)$.

\item $(\aR, \bR)\in\IPF(m, n)$ if and only if $\aR\in\PF(m, n)$ and $\tauR(\aR) \le_C \bR$.
\end{enumerate}
\end{proposition}

\begin{proof}
These equivalences follow directly from the definition.
\end{proof}

\begin{proposition}
Let $\cR=(\aR, \bR) \in \IPF(m, n)$. Then
\begin{enumerate}
\item $\bR^* \in \PF(m, n)$.

\item $\aR \leq_C \tauR(\cR) \leq_C \bR$ and $\tauR(\bR^*)^* \leq_C \bR$.
\end{enumerate}
\end{proposition}

\begin{proof}
Evidently $\aR \leq_C \tauR(\cR) \leq_C \bR$. Since $\tauR(\cR)$ is a parking outcome, it consists of distinct entries, and so its non-decreasing rearrangement $\lambdaR=(\lambda_1, \dots, \lambda_m)$ satisfies $\lambda_i \geq i$ for all $1\leq i\leq m$. It follows that $\tauR(\cR)^*$ also consists of distinct entries, and its non-decreasing rearrangement $\lambdaR^*=(\lambda_1^*, \dots, \lambda_m^*)=(n+1-\lambda_m, \dots, n+1-\lambda_1)$ satisfies $\lambda_i^* \leq n-m+i$ for all $1\leq i\leq m$. Therefore $\tauR(\cR)^* \in \PF(m, n)$. From $\tauR(\cR) \leq_C \bR$, one has $\bR^* \leq_C \tauR(\cR)^*$. Hence $\bR^* \in \PF(m, n)$. This implies that $\bR^* \leq_C \tauR(\bR^*)$, and further implies that $\tauR(\bR^*)^* \leq_C \bR$.
\end{proof}

\begin{proposition}
The number of interval parking functions $|\IPF(m, n)|$ satisfies
\begin{equation}
|\IPF(m, n)|=\sum_{\sR \models m} \binom{m}{\sR} \prod_{i=1}^{n-m+1} (s_i+1)^{s_i-1} \frac{n!}{\prod_{i=1}^{n-m} (n-i+1-s_1-\cdots-s_i)},
\end{equation}
where $\sR=(s_1, \dots, s_{n-m+1})$ is a composition of $m$. In particular,
\begin{equation} \label{count-IPF}
\left|\IPF(n, n)\right| = n!(n+1)^{n-1}.
\end{equation}
\end{proposition}

\begin{proof}
For an interval parking function $\cR=(\aR, \bR) \in \IPF(m, n)$, there are $n-m$ parking spots that are never attempted by any car. Let $k_i(\piR)$ for $i=1, \dots, n-m$ represent these spots, so that $0:=k_0<k_1<\cdots<k_{n-m}<k_{n-m+1}:=n+1$. This separates $\aR \in \PF(m, n)$ into $n-m+1$ disjoint non-interacting segments (some segments might be empty), with each segment a classical parking function of length $(k_{i}-k_{i-1}-1)$ after translation. The parking outcome is $\tauR(\cR)=\tauR(\aR)$, and for every $\aR \in \PF(m, n)$, there are precisely $n!/\prod_{i=1}^{n-m} (n-k_i+1)$ choices for $\bR$ such that $(\aR, \bR) \in \IPF(m, n)$. We have
\begin{align}
&|\IPF(m, n)|=\sum_{k} \prod_{i=1}^{n-m+1} (k_{i}-k_{i-1})^{k_{i}-k_{i-1}-2}\frac{n!}{\prod_{i=1}^{n-m} (n-k_i+1)} \cdot \notag \\
& \hspace{5cm} \cdot \binom{m}{k_1-k_0-1, \dots, k_{n-m+1}-k_{n-m}-1} \notag \\
&=\sum_{\sR \models m} \binom{m}{s_1, \dots, s_{n-m+1}} \prod_{i=1}^{n-m+1} (s_i+1)^{s_i-1} \frac{n!}{\prod_{i=1}^{n-m} (n-i+1-s_1-\cdots-s_i)},
\end{align}
where $\sR=(k_1-k_0-1, \dots, k_{n-m+1}-k_{n-m}-1)$ and $\sum_{i=1}^{n-m+1} s_i=m$.
\end{proof}

From (\ref{count-IPF}), we recognize that the number of interval parking functions $\IPF(n, n)$ coincides with the number of edge-labeled spanning trees of $K_{n+1}$. The rest of Section \ref{ipf} will focus on this combinatorial implication. We first present some background material on the symmetric group.

\subsection{The symmetric group as a Coxeter system}\label{Coxeter}

Denote by $\Sym_n$ the symmetric group on $n$ letters. We set $e=(1,\dots,n)$ (the identity permutation) and $w_0=(n,n-1,\dots,1)$. We denote by $t_{ij}$ the permutation transposing $i$ and $j$ and fixing all other values, and take $s_i=t_{i,i+1}$. The elements $s_1,\dots,s_{n-1}$ are termed the \textit{standard generators}. Our convention for multiplication is right to left, which is consistent with treating permutations as bijective functions from $[n]\to [n]$. Thus $t_{ij}x$ is obtained by transposing the digits $i,j$ wherever they appear in $x$, while $x t_{ij}$ is obtained by transposing the digits in the $i$th and $j$th positions.

The theory of normal forms in a Coxeter system was introduced by du~Cloux \cite{duCloux} and is elaborated in Bj\"{o}rner and Brenti \cite{BB}. The symmetric group $\Sym_n$ may be viewed as a Coxeter system of type A, with generators $S=\{s_1,\dots,s_{n-1}\}$. The \textit{length} $l(x)$ of $x\in\Sym_n$ is the smallest number $k$ such that $x$ can be written as a product $s_{i_1}\cdots s_{i_k}$ of standard generators; in this case $s_{i_1}\cdots s_{i_k}$ is called a \textit{reduced word} for $x$. It is a standard fact that length equals number of inversions:
\begin{equation} \label{length-inv}
\l(x)=\{(i,j):\ 1\leq i<j\leq n,\ x(i)>x(j)\}.
\end{equation}
Let $\sigma_k=s_k\cdots s_1$. Every $x\in\Sym_n$ has a unique normal form: a reduced word $N(x)$ of the form $v_1 \cdots v_{n-1}$, where $v_k=e$ or $v_k=s_k \cdots s_j$ for some $1\leq j\leq k$ is a prefix of $\sigma_k$. For example, $l(e)=0$, $N(e)=e$ and $l(w_0)=n(n-1)/2$, $N(w_0)=\sigma_1\cdots\sigma_{n-1}$. It is straightforward to obtain the permutation $x$ given its normal form $N(x)$. Conversely, since $xv_{n-1}^{-1} \cdots v_1^{-1}=e$, we may interpret the normal form decomposition of $x$ in an alternative way: Start with the permutation $x$. $v_{n-1}^{-1}$ corresponds to a sequence of adjacent transpositions that moves the value $n$ in $x$ to the right until it is in the last position (if $n$ is already in the last position then $v_{n-1}^{-1}=e$). Similarly, $v_{n-2}^{-1}$ corresponds to a sequence of adjacent transpositions that moves the value $n-1$ in $xv_{n-1}^{-1}$ to the right until it is in the next to last position (if $n-1$ is already in the next to last position then $v_{n-2}^{-1}=e$). And so on. Thus $x$ is fully characterized by the sequence
\begin{equation}
\lambdaR(x)=(\lambda_1(x),\dots,\lambda_{n-1}(x))=(|v_1|,\dots,|v_{n-1}|)\in[0,1]\times\cdots\times[0,n-1].
\end{equation}
This describes an explicit bijection between $\Sym_n$ and $C_2 \times \cdots \times C_n$, where $C_i$ is a chain with $i$ elements.

\subsection{One-to-one correspondence between interval parking functions $\IPF(n, n)$ and edge-labeled spanning trees of $K_{n+1}$}\label{1-1}

Recall the classical result that there exists a bijection between parking functions $\PF(n, n)$ and spanning trees of $K_{n+1}$, using the concept of \emph{specification} and \emph{order permutation}. Building upon this result, we will construct a bijection between interval parking functions $\IPF(n, n)$ and edge-labeled spanning trees of $K_{n+1}$, where the vertices are labeled $0$ through $n$ (vertex $0$ is the root) and the edges are labeled $1$ through $n$.

As illustrated in Chassaing and Marckert \cite{CM} and Yan \cite{Yan}, a parking function $\piR \in \PF(n, n)$ may be uniquely determined by its associated specification $\rR(\piR)$ and order permutation $\sigma(\piR)$. Here the specification is $\rR(\piR)=(r_1, \dots, r_n)$, where $r_k=\#\{i: \pi_i=k\}$ records the number of cars whose first preference is spot $k$. The order permutation $\sigma(\piR) \in \Sym_n$, on the other hand, is defined by
\begin{equation}
\sigma_i=|\{j: \pi_j<\pi_i, \text{ or } \pi_j=\pi_i \text{ and } j\leq i\}|,
\end{equation}
and so is the permutation that orders the list, without switching elements which are the same. In words, $\sigma_i$ is the position of the entry $\pi_i$ in the non-decreasing rearrangement of $\piR$. Conversely, we can easily recover a parking function $\piR$ by replacing $i$ in $\sigma(\piR)$ with the $i$th smallest term in the sequence $1^{r_1}\dots n^{r_n}$.

However, not every pair of a length $n$ vector $\rR$ and a permutation $\sigma \in \Sym_n$ can be the specification and the order permutation of a parking function from $\PF(n, n)$. The vector and the permutation must be compatible with each other, in the sense that the terms $1+\sum_{i=1}^{k-1} r_i, \dots, \sum_{i=1}^{k} r_i$ appear from left to right in $\sigma$ for every $k$ to satisfy the non-decreasing rearrangement requirement of $\piR$. Moreover, the specification $\rR$ should satisfy a balance condition:
\begin{align}\label{determine}
\sum_{s=1}^{j} r_s \geq j, \hspace{.2cm} \forall 1 \leq j \leq n, \hspace{1cm} \sum_{s=1}^n r_s=n.
\end{align}
Let $\C(n)$ be the set of all compatible pairs.

Denote by $\F(n+1)$ the set of spanning trees of $K_{n+1}$, where the vertices are labeled $0$ through $n$ and vertex $0$ is the root. Further denote by $\F^e(n+1)$ the set of edge-labeled spanning trees of $K_{n+1}$, where the edges, in addition to the vertices, are also labeled $1$ through $n$.

\begin{theorem}[adapted from Yan \cite{Yan}]\label{Yan}
The set $\C(n)$ is in one-to-one correspondence with $\PF(n, n)$, and is also in one-to-one correspondence with $\F(n+1)$.
\end{theorem}

\begin{theorem}\label{edge-labeled}
There is a one-to-one correspondence between $\IPF(n, n)$ and $\F^e(n+1)$, the set of edge-labeled spanning trees of $K_{n+1}$.
\end{theorem}

\begin{proof}
By Proposition \ref{IPF-equivalence}, $(\aR, \bR) \in \IPF(n, n)$ is equivalent to $\aR\in \PF(n, n)$ and $\tauR(\aR) \leq_C \bR$. Using Theorem \ref{Yan}, $\aR$ is in one-to-one correspondence with a spanning tree of $K_{n+1}$, where $\aR$ determines the shape and vertex labels of the spanning tree. Since $\tauR(\aR)$ is a permutation on $n$ letters, $\bR-\tauR(\aR)$ takes values in $C_1 \times \cdots \times C_n$, where $C_i$ is a chain of length $i$ (after reordering the indices). Using results on Coxeter systems from Section \ref{Coxeter}, this gives an association between $\bR$ and the edge labels of the spanning tree.

\begin{figure}
\begin{center}
\begin{forest}
for tree={circle, draw, l sep=3em}[0, baseline, fill={gray}[2, root color={red}, edge label={node[midway,left] {5}}[6, root color={red}, edge label={node[midway,left] {3}}[7, root color={red}, edge label={node[midway,left] {2}}]]][5, root color={red}, edge label={node[midway,right] {6}}[1, root color={red}, edge label={node[midway,left] {4}}][8, root color={red}, edge label={node[midway,right] {1}}[3, root color={red}, edge label={node[midway,right] {7}}]]][9, root color={red}, edge label={node[midway,right] {9}}[4, root color={red}, edge label={node[midway,right] {8}}]]]
\end{forest}
\caption{Edge-labeled spanning tree of complete graph.}\label{illustration1}
\end{center}
\end{figure}
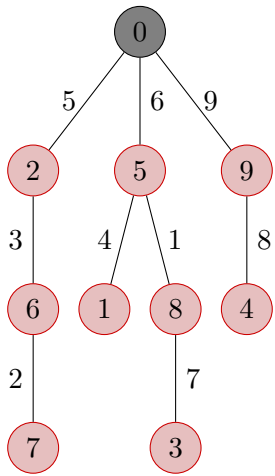

We illustrate the map with a representative example. See Figure \ref{illustration1} representing an element of $\F^e(10)$. We read the vertices in ``breadth first search" (BFS) order: $v_{0}, \dots, v_{9}=0,2,5,9,6,1,8,4,7,3$. That is, read the root vertex first, then all vertices at level one (distance one from the root), then those at level two (distance two from the root), and so on, where vertices at a given level are naturally ordered in order of increasing predecessor, and, if they have the same predecessor, increasing order. We let $\sigma=259618473$ be this vertex ordering once we remove the root vertex. We also record the edges incident with the vertices as $x=569341827$, with associated normal form $\lambdaR(x)=(0, 2, 2, 4, 4, 0, 2, 6) \in C_2 \times \cdots \times C_9$. We let $r_i$ record the number of successors of $v_i$, that is, $\rR=(3, 1, 2, 1, 1, 0, 1, 0, 0)$. Now $\rR$ is balanced and $\sigma^{-1}=519724863$ is compatible with $\rR$, by virtue of the fact that vertices with the same predecessor are read in increasing order. The corresponding parking function is $\aR=(3, 1, 7, 4, 1, 2, 5, 3, 1)$, with parking outcome $\tauR(\aR)=(3, 1, 7, 4, 2, 5, 6, 8, 9)$. Thus $\bR-\tauR(\aR) \in C_7 \times C_9 \times C_3 \times C_6 \times C_8 \times C_5 \times C_4 \times C_2 \times C_1$. Reordering the indices in $\lambdaR(x)$ and adding an extra $0$ (for $C_1$), we have $\bR-\tauR(\aR)=(0, 6, 2, 4, 2, 4, 2, 0, 0)$. Hence $\bR=(3, 7, 9, 8, 4, 9, 8, 8, 9)$. The interval parking function connected with this edge-labeled spanning tree is $\cR=(\aR, \bR)=((3, 1, 7, 4, 1, 2, 5, 3, 1), (3, 7, 9, 8, 4, 9, 8, 8, 9))$.

The above one-to-one correspondence between edge-labeled spanning trees and interval parking functions does not depend on using the BFS algorithm; any other algorithm which builds up a tree one edge at a time through a sequence of growing subtrees will give an alternate bijection. Generally, an algorithm checks the vertices of the tree one-by-one, starting with the root. At each step, we pick a new vertex and connect it to the checked vertices. The choice function (which defines the algorithm) tells us which new vertex to pick.
\end{proof}

Equivalently, we could view the edge-labeled spanning tree of $K_{n+1}$ as the spanning tree of a complete bipartite graph of $K_{n, n+1}$ where the first group has $n$ vertices labeled $1$ through $n$ and the second group has $n+1$ vertices labeled $0$ through $n$, and every vertex in the first group has two incident edges. Two vertices $i$ and $j$ of $K_{n+1}$ are connected with edge label $k$ if and only if vertices $i$ and $j$ in the second group of $K_{n, n+1}$ are both connected to vertex $k$ in the first group. This is a one-to-one correspondence, since vertex $k$ must be unique as otherwise this creates a cycle in $K_{n, n+1}$. See Figure \ref{illustration2} for a transformed view of Figure \ref{illustration1}.

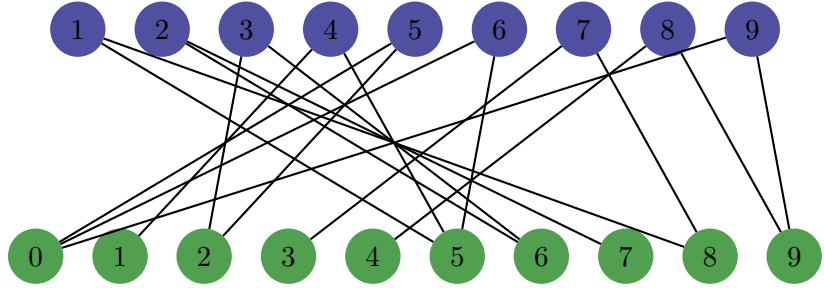
\begin{figure}
\begin{center}
\begin{tikzpicture}[thick,amat/.style={matrix of nodes,
  row sep=1em, column sep=1em,
  nodes={anchor=center,align=center,text width=1.8em,inner sep=0pt,outer sep=0pt,solid,circle}},
  fsnode/.style={fill=myblue},
  ssnode/.style={fill=mygreen}]

 \matrix[amat,nodes=fsnode] (mat1) {1&2&3&4&5&6&7&8&9\\};

 \matrix[amat,below=2cm of mat1,nodes=ssnode] (mat2) {0&1&2&3&4&5&6&7&8&9\\};

 \draw  (mat1-1-1) edge (mat2-1-6)
  (mat1-1-1) edge (mat2-1-9)
  (mat1-1-2) edge (mat2-1-7)
  (mat1-1-2) edge (mat2-1-8)
  (mat1-1-3) edge (mat2-1-3)
  (mat1-1-3) edge (mat2-1-7)
  (mat1-1-4) edge (mat2-1-2)
  (mat1-1-4) edge (mat2-1-6)
  (mat1-1-5) edge (mat2-1-1)
  (mat1-1-5) edge (mat2-1-3)
  (mat1-1-6) edge (mat2-1-1)
  (mat1-1-6) edge (mat2-1-6)
  (mat1-1-7) edge (mat2-1-4)
  (mat1-1-7) edge (mat2-1-9)
  (mat1-1-8) edge (mat2-1-5)
  (mat1-1-8) edge (mat2-1-10)
  (mat1-1-9) edge (mat2-1-1)
  (mat1-1-9) edge (mat2-1-10);
\end{tikzpicture}
\caption{Spanning tree of complete bipartite graph.}\label{illustration2}
\end{center}
\end{figure}

\section*{Acknowledgements}

Mei Yin acknowledges helpful conversations with Jeremy L. Martin, and is particularly thankful to Richard Kenyon for many enlightening comments.

\end{document}